\documentclass[final]{siamltex}
\usepackage{amsmath,amsfonts,amssymb}
\newtheorem{remark}{Remark}
\newtheorem{assumption}{Assumption}

\usepackage{latexsym}
\usepackage{epsfig,overpic}
\usepackage{todonotes}
\usepackage{booktabs}
\usepackage{hyperref}
\hypersetup{
  colorlinks= true,
  allcolors = {green!50!black},
  urlcolor  = {red!50!black},
}

\usepackage{pgfplots}
\usetikzlibrary{shapes,arrows,snakes,calendar,matrix,backgrounds,folding,calc,positioning,spy}
\usepackage{tikz}

\usepackage{siunitx}
\newcommand{\R}{\mathbb{R}}
\newcommand{\mP}{\mathcal{P}}
\newcommand{\tr}{\operatorname{tr}}
\renewcommand{\d}[1]{\,\mathrm{d} #1}
\DeclareMathOperator{\meas}{\operatorname{meas}}
\newcommand{\norm}[1]{\left\lVert #1 \right\rVert}
\newcommand{\enormh}[1]{\Vert #1 \Vert_h}
\newcommand{\rr}{\mathbb{R}}

\newcommand{\bff}{\mathbf{f}}
\newcommand{\bu}{\mathbf{u}}

\newcommand{\bc}{\mathbf{c}}
\newcommand{\be}{\mathbf{e}}
\newcommand{\bM}{\mathbf{M}}

\newcommand{\mT}{\mathcal{T}}
\newcommand{\mF}{\mathcal{F}}
\newcommand{\mI}{\mathcal{I}}

\newcommand{\Vtracez}{V_{h,\Theta}^{\Gamma,0}}
\newcommand{\Vk}{V_{h,\Theta}^k}
\newcommand{\Vkz}{V_{h,\Theta}^{k,0}}

\newcommand{\Ik}{I^k}
\newcommand{\Ione}{I^1}
\newcommand{\ITk}{I_{\Theta}^k}
\newcommand{\ITkz}{I_{\Theta}^{k,0}}

\newcommand{\mTGamma}{\mT^{\Gamma}}
\newcommand{\mFGamma}{\mF^{\Gamma}}

\newcommand{\lin}{\text{lin}}
\newcommand{\Gammalin}{\Gamma^{\lin}}
\newcommand{\jump}[1]{[\![#1]\!]}

\newcommand{\OGamma}{\Omega^{\Gamma}}

\newcommand{\hphi}{\hat \phi_h}

\newcommand{\Vregh}{V_{\text{reg},h}}

\newcommand{\nlin}{n_{\rm lin}}

 \newcommand{\bv}{\mathbf v}
\newcommand{\bS}{\mathbf S}
\newcommand{\averageleft}{\{\!\!\{}
\newcommand{\averageright}{\}\!\!\}}
\newcommand{\average}[1]{\averageleft #1 \averageright}

 \newcommand{\TGamma}{{\mathcal T}^{\Gamma}}
\DeclareSIUnit[number-unit-product = {~}]\Q{~}
\DeclareSIPrefix\kilo{K}{3}
\DeclareSIPrefix\mega{M}{6}
\DeclareSIPrefix\terra{K}{9}

\title{Analysis of a high order Trace Finite Element Method for PDEs on level set surfaces}
\author{
J\"org Grande\thanks{Institut f\"ur
Geometrie und Praktische Mathematik, RWTH Aachen University, D-52056 Aachen,
Germany; email: {\tt \{grande,reusken\}@igpm.rwth-aachen.de}}
\and
Christoph Lehrenfeld\thanks{Institut f\"ur Numerische und Angewandte Mathematik,
Georg-August-Universit\"at G\"ottingen, D-37083 G\"ottingen, Germany; email: {\tt lehrenfeld@math.uni-goettingen.de}}
\and
Arnold Reusken\footnotemark[1]
}

\begin{document}
\maketitle
\begin{abstract}
We present a new high order finite element method for the discretization of partial differential equations on stationary smooth surfaces which are implicitly described as the zero level of a level set function. The discretization is based on a trace finite element technique. The higher discretization accuracy is obtained by using an isoparametric mapping of the volume mesh, based on the level set function, as introduced in [C. Lehrenfeld, \emph{High order unfitted finite element methods on level set domains using isoparametric mappings}, Comp. Meth. Appl. Mech. Engrg. 2016]. The resulting trace finite element method is easy to implement. We present an error analysis of this method and derive optimal order $H^1(\Gamma)$-norm error bounds. A second topic of this paper is a unified analysis of several stabilization methods for trace finite element methods. % Three methods known from the literature and one recently developed method are analyzed in a general framework.
Only a stabilization method which is based on adding an anisotropic diffusion in the volume mesh is able to control the condition number of the stiffness matrix also for the case of higher order discretizations. Results of numerical experiments are included which confirm the theoretical findings on optimal order discretization errors and uniformly bounded condition numbers.
\end{abstract}


%\tableofcontents

\section[Introduction]{Introduction}\label{sec:introduction}
% \subsection{Problem statement / Motivation}
% {\bf essentially copy+pasted...: needs a refactorization}
Recently there has been an increasing interest in \emph{unfitted finite element methods}. These methods offer the possibility to handle complex geometries which are not aligned with a computational (background) mesh. Also the development and analysis of numerical methods for PDEs on (evolving) surfaces is a rapidly growing research area.

The trace finite element method (TraceFEM) \cite{olshanskii2009finite} is an unfitted FEM for PDEs on implicit domains which are described via a level set function. 
In this paper we introduce and analyze a \emph{higher order} TraceFEM for surface PDEs. Furthermore, several \emph{stabilization methods} are studied. The aim of these methods is to obtain condition numbers which are uniformly bounded with respect to the location of the surface in the underlying volume triangulation. 

\subsection*{Literature}
The TraceFEM was introduced in \cite{olshanskii2009finite} for elliptic PDEs on smooth stationary surfaces. 
For piecewise linears, the method has been studied extensively. For stationary surfaces, the conditioning properties of the resulting stiffness matrices are discussed in \cite{reusken10matrixprop}. 
Convection dominated problems are considered in \cite{olshanskii14advection,burman15cutdgimajna}. In \cite{olshanskii14advection} a streamline diffusion stabilization is treated, whereas in \cite{burman15cutdgimajna} a Discontinuous Galerkin formulation is studied.
For PDEs on \emph{evolving} surfaces, space-time formulations of the TraceFEM were first considered in \cite{grande2014movingsurfaces}. A space-time formulation of the TraceFEM is analyzed in \cite{grande2014spacetime,olshanskii14spacetime,olshanskii14spacetimeanalysis} .
In all these publications, only piecewise linears are considered. 

One major issue in the design and realization of high order methods in the context of unfitted finite element methods is the problem of numerical integration on domains which are represented implicitly. Different approaches to deal with this issue exist, cf. the literature overview in \cite{lehrenfeld2015cmame}. 

For surface PDEs on implicit domains, higher order FE methods have first been considered in \cite{demlow2009higher}. In that paper it is crucial that the surface is given as the \emph{zero level of a smooth signed distance function  which is explicitly known}. Based on this distance function a parametric mapping of a shape regular piecewise triangular surface approximation to the zero level of the distance function is constructed which results in a higher order surface representation. In that method the finite element space is explicitly defined with respect to this triangular surface approximation. Hence, it is not a TraceFEM.
In \cite{grande2014highordersjna} a higher order TraceFEM discretization is introduced for PDEs on surfaces which are represented as the zero level of a level set function, which is not necessarily a signed distance function. To this end, 
a \emph{parametric mapping of a piecewise planar interface approximation} 
is constructed based on quasi-normal fields. Both in \cite{demlow2009higher} and \cite{grande2014highordersjna} optimal a-priori error bounds are derived. 
An approach, similar to the one in \cite{grande2014highordersjna}, to enhance the geometry approximation of a piecewise planar interface approximation has recently been introduced in \cite{lehrenfeld2015cmame}. In the latter paper, however, a \emph{parametric mapping of the underlying mesh} is used. The construction of such a mapping is presented in \cite{lehrenfeld2015cmame}, and optimal approximation properties have been derived in \cite{CLARH1} for an elliptic interface model problem. 
The parametric mapping of the underlying mesh allows for a high order approximation of both bulk domains and implicitly defined surfaces/interfaces. Hence this approach can be used to obtain higher order discretizations for interface problems (as in \cite{CLARH1}) as well as for surface-bulk coupled problems with Trace FEM (as considered in \cite{goresaim}).

Different aspects, which are less relevant for the topic of this paper, of high order discretizations on \emph{triangulated surface} are treated in \cite{demlow2009higher,langer2014discontinuousbook,antonietti2015high}.

Related to the conditioning of stiffness matrices in the TraceFEM, we note the following. To improve the conditioning of the stiffness matrices in the TraceFEM, the ``full gradient volume stabilization'' has been considered in \cite{burman16fullgradcmame,reusken2015}. Other techniques known in the literature are the ``ghost penalty stabilization'' \cite{Burman2010,BHL15} and the ``full gradient surface stabilization'' \cite{DER14,reusken2015}. In this paper we study one further method which we call ``normal derivative volume stabilization''. This stabilization has also been proposed in the recent preprint \cite{burmanembedded} (which we were not aware of while studying the method).
As we will explain further on, this method outperforms the other three in the case of higher order trace finite element discretizations.
% While studying the normal derivative volume stabilization, this method appeared to be new. However, after submission of this paper it turned out that the same method had been developed in parallel by another research group and published in the recent preprint \cite{burmanembedded}.
The relation between the results on this stabilization method presented in this paper and in \cite{burmanembedded} is discussed in Remark~\ref{comextra}. A comparison of all four above-mentioned stabilization methods is given in section~\ref{sec:condition-number}.

\subsection*{Main contributions of this paper}
We use the approach presented in \cite{lehrenfeld2015cmame} to obtain a \emph{higher order isoparametric TraceFEM for surface PDEs}. The method needs as input only a (high order) finite element approximation of the level set function and is easy to implement (in particular, easier than the method treated in \cite{grande2014highordersjna}). In this TraceFEM a finite element space is defined 
on a \emph{transformed} background mesh and a discretization is obtained by restricting the corresponding functions to an (approximated) surface and applying a Galerkin formulation. The isoparametric mapping of the background mesh is the key ingredient for obtaining a higher order discretization, very similar to the standard finite element isoparametric technique for higher order boundary approximation. We present an error analysis for this method and derive \emph{optimal order $H^1$-norm discretization error bounds}. A second main contribution of this paper is concerned with stabilization methods for obtaining condition numbers which are uniformly bounded with respect to the location of the surface in the underlying volume triangulation. We present a \emph{unified framework} for analyzing such methods and treat the recently developed ``\emph{normal derivative volume stabilization}''. The analysis reveals that only the latter method is suitable for higher order trace finite element discretizations.
\subsection*{Structure of the paper}
In section \ref{sec:problem} we recall the weak formulation of the Laplace--Beltrami equation and introduce our assumptions concerning the geometry description based on a level set function. The parametric mapping used to obtain a high order accurate geometry description is introduced in section 
\ref{sec:mapping}. The isoparametric trace FEM is given in section \ref{sec:itracefem}. In that section we introduce a generic stabilization $s_h(\cdot,\cdot)$. In section \ref{sec:error-analysis}
we derive an optimal a-priori discretization error bound in the $H^1$-norm. For this we need two conditions on the stabilization bilinear form $s_h(\cdot,\cdot)$ to hold. 
In section \ref{sec:condition-number} we derive condition number bounds for the stiffness matrix which are robust with respect to the
position of the surface in the computational mesh. In this analysis a third condition for the stabilization bilinear form $s_h(\cdot,\cdot)$ is introduced. It is shown that the three conditions on $s_h(\cdot,\cdot)$ that arise in the analysis are satisfied for certain known stabilization methods applied to \emph{linear} FE discretizations. 
An analysis of  the normal derivative volume stabilization is given in section \ref{sec:new-stab}. This analysis shows that for this method the three conditions on $s_h(\cdot,\cdot)$ are satisfied also for \emph{higher order} trace finite element discretizations.
 Numerical experiments which illustrate the (optimal) higher order of convergence and the conditioning of the corresponding stiffness matrices are provided in section \ref{sec:numex}. A summary and outlook are given in section \ref{sec:conclusion}.

\section{Problem formulation} \label{sec:problem}
%\subsection{Weak formulation} \label{sec:weakform}
Let $\Omega\subseteq\R^3$
%{\bf AR: I restricted to the 3D case, for convenience}
be a polygonal domain and $\Gamma \subset \Omega $ a smooth, closed, connected 2D surface.
Given $f\in H^{-1}(\Gamma)$, with $f(1)=0$ we consider the following Laplace--Beltrami equation:
%\begin{equation}\label{eq:strong-LB}
%  -\Delta_\Gamma u + \rho u = f\quad\text{on }\Gamma
%\end{equation}
%in the sense of distributions. 
%If $\rho=0$, additional assumptions like $(u, 1)_\Gamma = 0$, $(f, 1)_\Gamma=0$ are needed for well-posedness. \\
Find $u\in H_\ast^1(\Gamma):= \{\, v \in H^1(\Gamma)~|~\int_\Gamma v \, \d{s}=0 \, \}$ such that
\begin{equation}\label{eq:weak-LB}
  a(u,v)= f(v)\quad\text{for all } v\in H_\ast^1(\Gamma)
\end{equation}
with
\[
  a(u,v)= \int_\Gamma \nabla_\Gamma u \cdot \nabla_\Gamma v \, \d{s}.
\]

\subsection{Geometry description through a level set function} \label{prelim}
We assume that the smooth surface $\Gamma$ is the zero level of a smooth level set function $\phi$, i.e., $\Gamma= \{\, x \in \Omega~|~\phi(x)=0\,\}$.
This level set function is not necessarily close to a distance function, but has the usual properties of a level set function: 
\begin{equation} \label{eq:lsetprops}
\|\nabla \phi(x)\| \sim 1~,~~\|D^2 \phi(x)\| \leq c \quad \text{for all}~x~~ \text{in a neighborhood $U$ of $\Gamma$}.
\end{equation}
We assume that the level set function has the smoothness property $\phi \in C^{k+2}(U)$, where $k$ denotes the polynomial degree of the finite element space introduced below. The assumptions on the level set function \eqref{eq:lsetprops} imply the following relation, which is fundamental in the analysis below:
\begin{equation} \label{eq:lsetdist}
 |\phi(x+\epsilon \nabla \phi(x)) - \phi(x+\tilde \epsilon \nabla \phi(x))| \sim |\epsilon - \tilde \epsilon|, \quad x \in U,
\end{equation}
for $|\epsilon|, |\tilde \epsilon|$ sufficiently small.

% As input for the parametric mapping we need an approximation $\phi_h \in V_h^k$ of $\phi$, and we assume that this approximation satisfies the error estimate
% \begin{equation} \label{eq:lsetapprox}
% \|\phi_h - \phi\|_{\infty, U} + h\|\nabla(\phi_h - \phi)\|_{\infty, U} \lesssim h^{k+1}. 
% \end{equation}
We assume a simplicial triangulation of $\Omega$, denoted by $\mT \in \{\mT_h\}_{h>0}$, and the standard finite element space of continuous piecewise polynomials up to degree $k$ by $V_h^k$. 
The nodal interpolation operator in $V_h^k$ is denoted by $\Ik$. 

For ease of presentation we assume quasi-uniformity of the mesh, and $h$ denotes a characteristic mesh size with $h\sim h_T:={\rm diam}(T)$, $~T\in\mT$.

As input for the parametric mapping we need an approximation $\phi_h \in V_h^k$ of $\phi$, and in the error analysis we assume that this approximation satisfies the error estimate
\begin{equation} \label{err2}
  \max_{T\in \mT} |\phi_h - \phi|_{m,\infty,T \cap U} \lesssim h^{k+1-m},\quad 0 \leq m \leq k+1.
\end{equation}
Here $|\cdot|_{m,\infty,T\cap U}$ denotes the usual semi-norm on the Sobolev space $H^{m,\infty}(T\cap U)$ and the constant used in $\lesssim$ depends on $\phi$ but is independent of $h$.
Note that \eqref{err2} implies the estimate
\begin{equation} \label{err1}
 \|\phi_h - \phi\|_{\infty, U} + h\|\nabla(\phi_h - \phi)\|_{\infty, U} \lesssim h^{k+1}. 
\end{equation}
The zero level of the finite element function $\phi_h$ (implicitly) characterizes an approximation of the interface. 
\emph{The piecewise linear nodal interpolation of $\phi_h$ is denoted by $\hphi = \Ione \phi_h$.} Hence, $\hphi(x_i)=\phi_h(x_i)$ at all vertices $x_i$ in the triangulation $\mT$. 
The low order geometry approximation of the interface, which is needed in our discretization method, is the zero level of this function:
\[
  \Gammalin := \{ x\in\Omega\mid \hphi(x) = 0\}.
\]
All elements in the triangulation $\mT$ which are cut by $\Gammalin$ are collected in the set $\mTGamma := \{T \in \mT\mid T \cap \Gammalin \neq \emptyset \}$. The corresponding domain is $\OGamma := \{ x \in T \mid T\in \mTGamma\}$. 
We define the set of facets \emph{inside} $\Omega^\Gamma$, $\mathcal{F}^\Gamma := \{ F = \overline{T}_a \cap \overline{T}_b; T_a,T_b \in \mT^\Gamma, \meas_{2}(F) > 0\}$.

\section{The isoparametric mapping} \label{sec:mapping}
We use the mesh transformation introduced in \cite{lehrenfeld2015cmame} and \cite{CLARH1}. 
We only outline the important ingredients in the construction of the mapping. For details we refer to the thorough discussion in \cite[Section 3]{CLARH1}. 

We first introduce a mapping $\Psi$ on $\OGamma$ with the property $\Psi(\Gammalin)=\Gamma$. Using $G:=\nabla \phi$ a function $d: \OGamma \to \mathbb{R}$ is defined as follows: $d(x)$ is the (in absolute value) smallest number such that
\begin{equation} \label{cond1}
  \phi(x + d(x) G(x))=\hat \phi_h(x) \quad \text{for}~~x \in \OGamma.
\end{equation}
Let $C^l(\TGamma) := \{ v \mid v|_T \in C^l(T),~ T \in \TGamma\},~l \in \mathbb{N}_0,$ be the space of element-wise $C^l$-continuous functions that can be discontinuous across element faces.
In \cite{CLARH1} it is shown that for $h$ sufficiently small the relation in \eqref{cond1} defines a unique 
 $d(x)$ and $d \in C(\OGamma) \cap H^{1,\infty} (\OGamma) \cap C^{k+1}(\TGamma)$. Given the function
 $dG \in [C(\OGamma)]^3 \cap [H^{1,\infty}(\OGamma)]^3$ we define:
\begin{equation} \label{psi1}
 \Psi(x):= x + d(x) G(x), \quad x \in \OGamma.
\end{equation}
In general, e.g., if $\phi$ is not explicitly known, the mapping $\Psi$ is not computable. We introduce an easy to construct accurate approximation of $\Psi$ as follows.
% Note that $I(\hat \phi_h,\phi_h)$ is uniquely defined and $ I(\hat \phi_h,\phi_h)=\hat \phi_h$ on $\OGamma$, $I(\hat \phi_h,\phi_h)= \phi_h $ on $\partial \OGammaplus$. 
% Let $\mathcal{E}_T \phi_h$ be the polynomial extension of $\phi_h|_T$. We define a function $d_h: \TGammaplus \to [-\delta,\delta]
% $, with $\delta > 0$ sufficiently small, as follows: $d_h(x)$ is the (in absolute value) smallest number such that 
%   \begin{equation} \label{eq:psihmap}
%     \mathcal{E}_T \phi_h(x + d_h(x) G_h(x)) = I(\hat \phi_h, \phi_h)(x), \quad \text{for}~~ x\in T \in \TGammaplus.
%  \end{equation}
% Hence, at $x \in T$, the scalar function $d_h(x)$ is a steplength in the direction $G_h(x)$ such that the value of $ \mathcal{E}_T \phi_h$ at $x + d_h(x) G_h(x)$ coincides with $I(\hat \phi_h, \phi_h)(x)$. In \cite{CLARH1} it is shown that, for $h$ sufficiently small, the relation \eqref{eq:psihmap} defines a unique $d_h(x)$. Furtermore, $d_h(x) =0 $ for all $x \in\partial \OGammaplus$, hence $d_h$ can be extend by zero on $\Omega \setminus \OGammaplus$.
% Given the function $d_h G_h \in C(\T)^d$ we define
%  \begin{equation} \label{eq:psih}
%     \Psi_h(x) := x + d_h(x) G_h(x) \quad \text{for}~x \in \Omega, ~~ \Theta_h^\Gamma:= P_h \Psi_h = \id + P_h(d_h G_h).
%   \end{equation}
% The mappings $\Theta_h^\Gamma$ and $\Theta_h^{\partial \Omega}$ can be determined with very low computational costs. Implementation aspects related to the construction of $\Theta_h^\Gamma$ and shape regularity properties are discussed in \cite{lehrenfeld15,CLARH1}.
%
%

Let $\mP(T)$ and $\mP(\mathbb{R}^3)$ denote the space of polynomials of degree up to $k$ on $T$ and $\mathbb{R}^3$, respectively. We define the polynomial extension 
$\mathcal{E}_T: \mP(T) \rightarrow \mP(\mathbb{R}^3)$ so that for $v \in V_h^k$ we have $(\mathcal{E}_T v)|_T = v|_T,~T\in\TGamma$.
For a search direction $G_h$ we need a sufficiently accurate approximation of $\nabla \phi$. In this paper we take 
\[ G_h = \nabla \phi_h,
\]
but there are other options.
Given $G_h$ we define a function $d_h: \mTGamma \to [-\delta,\delta]
$, with $\delta > 0$ sufficiently small, as follows: 
$d_h(x)$ is the (in absolute value) smallest number such that
\begin{equation} \label{eq:psihmap}
  \mathcal{E}_T \phi_h(x + d_h(x) G_h(x)) = \hphi(x), \quad \text{for}~~ x\in T \in \mTGamma.
\end{equation}
In the same spirit as above, corresponding to $d_h$ we define
\begin{equation} \label{eq:psih}
  \Psi_h(x) := x + d_h(x) G_h(x), \quad \text{for}~~ x\in T \in \mTGamma,
\end{equation}
which is an approximation of the mapping $\Psi$ in \eqref{psi1}.
The mapping may be discontinuous across facets. We use a simple and well-known projection operator, cf. \cite[Eqs.(25)-(26)]{oswald} and \cite{ernguermond15}, to map $\Psi_h$ into the finite element space. This projection relies on a nodal representation of the finite element space $(V_h^k)_{|\OGamma}$.
%However, similar operators can also be defined for different bases.
The set of finite element nodes $x_i$ in $\TGamma$ is denoted by $N(\TGamma)$, and $N(T)$ denotes the set of finite element nodes associated to $T \in \TGamma $. All elements $T \in \TGamma$ which contain the same finite element node $x_i$ form the set denoted by $\omega(x_i)$, i.e., 
$
\omega(x_i)  := \{ T \in \TGamma | x_i \in N(T) \}, ~ x_i \in N(\TGamma).
$
For each finite element node we define the local average as 
\[
  A_{x_i}(v):= \frac{1}{|\omega(x_i)|} \sum_{T \in \omega(x_i)} v_{|T}(x_i), ~~x_i \in N(\TGamma).
\]
where $|\cdot|$ denotes the cardinality of the set $\omega(x_i)$.
The projection operator $P_h: C(\TGamma)^3 \to (V_h^k)^3$ is defined as
\begin{equation} \label{defP} 
P_h v:= \sum_{x_i \in N(\TGamma)} A_{x_i}(v) \, \psi_i, \quad v \in C(\TGamma)^d,
\end{equation}
where $\psi_i$ is the nodal basis function corresponding to $x_i$. Note that
\[ P_h w = \Ik w ~~\forall~ w \in C(\OGamma)^3, \quad  \|P_h w\|_{\infty,\Omega} \lesssim \max_{T \in\TGamma} \| w\|_{\infty,T}~~\forall~ w \in C(\TGamma)^3.
\]
Based on this projection operator we define
\begin{equation} \label{defThetah} \Theta_h := P_h \Psi_h \in (V_h^k)^3,
\end{equation} 
Note that $k=1$ implies $\phi_h= \hat \phi_h$ and thus $\Theta_h ={\rm id}$. Hence, only $k \geq 2$ is of interest. 
Based on the transformation $\Theta_h$ we define
\begin{equation}
\Gamma_h := \Theta_h(\Gammalin) = \{ x~|~\hat\phi_h \big( \Theta_h^{-1}(x)\big) = 0 \}.
\end{equation}
 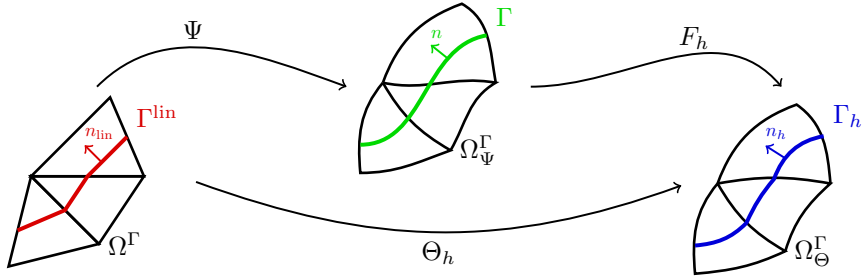
\begin{figure}[h!]
   \vspace*{-0.2cm}
   \begin{center}
     \begin{tikzpicture}[scale=1.5]
       \node(Olin) {
     \begin{tikzpicture}[scale=1.5]
       \coordinate (A) at (0.0,0.0);
       \coordinate (B) at (0.7,0.7);
       \coordinate (C) at (1.0,0.0);
       \coordinate (D) at (0.6,-0.6);
       \coordinate (E) at (-0.2,-0.8);
       \coordinate (F) at ($(A)!0.6!(E)$);
       \coordinate (G) at ($(A)!0.5!(D)$);
       \coordinate (H) at ($(A)!0.5!(C)$);
       \coordinate (I) at ($(B)!0.5!(C)$);
       \draw[line width=1pt] (A) -- (B) -- (C) --cycle;
       \draw[line width=1pt] (A) -- (C) -- (D) --cycle;
       \draw[line width=1pt] (A) -- (D) -- (E) --cycle;
       \draw[red!85!black,line width=1.5pt] (F) -- (G) -- (H) -- (I);
       \node[above right] at (I) {\color{red!85!black}$\Gammalin$};
       \node[right] at (D) {\color{black}$\OGamma$};
       \coordinate (J) at (0.625,0.125);
       \coordinate (K) at (0.475,0.275);
       \draw[red!85!black,line width=0.75pt,->] (J) -- (K);
       \node[above right,scale=0.7] at (K) {\!\!\color{red!85!black}$n_{\text{lin}}$};
     \end{tikzpicture}
   };
   \node[right= 2cm of Olin.north east](Oexa) {
     \begin{tikzpicture}[scale=1.5]
       \coordinate (A) at (0.0,0.0);
       \coordinate (B) at (0.7,0.7);
       \coordinate (C) at (1.0,0.0);
       \coordinate (D) at (0.6,-0.6);
       \coordinate (E) at (-0.2,-0.8);
       \coordinate (F) at (-0.2,-0.55);
       \coordinate (I) at (0.92,0.42);
       \draw[line width=1pt] (A) to[in=-160,out=70] (B) to[in=100,out=-30] (C);
       \draw[line width=1pt] (A) to[in=170,out=-12] (C) to[in=60,out=-150] (D);
       \draw[line width=1pt] (A) to[in=150,out=-60] (D) to[in=0,out=-160] (E) to[in=-130,out=95] (A);
       \draw[green!85!black,line width=1.5pt] (F) to[in=-170,out=0] (I);
       \node[above right] at (I) {\color{green!85!black}$\Gamma$};
       \node[right] at (D) {\color{black}$\OGamma_{\Psi}$};
       \coordinate (J) at (0.575,0.215);
       \coordinate (K) at (0.425,0.335);
       \draw[green!85!black,line width=0.75pt,->] (J) -- (K);
       \node[above right,scale=0.7] at (K) {\!\!\color{green!85!black}$n$};
     \end{tikzpicture}
   };
   \node[right = 2cm of Oexa.east, anchor=north west](Oh) {
     \begin{tikzpicture}[scale=1.5]
       \coordinate (A) at (0.0,0.0);
       \coordinate (B) at (0.7,0.7);
       \coordinate (C) at (1.0,0.0);
       \coordinate (D) at (0.6,-0.6);
       \coordinate (E) at (-0.2,-0.8);
       \coordinate (F) at (-0.2,-0.55);
       \coordinate (G) at (0.255,-0.35);
       \coordinate (H) at (0.5,0.04);
       \coordinate (I) at (0.94,0.42);
       \draw[line width=1pt] (A) to[in=-160,out=70] (B) to[in=95,out=-35] (C);
       \draw[line width=1pt] (A) to[in=170,out=10] (C) to[in=73,out=-138] (D);
       \draw[line width=1pt] (A) to[in=150,out=-60] (D) to[in=5,out=-155] (E) to[in=-130,out=95] (A);
       \draw[blue!85!black,line width=1.5pt] (F) to[in=-135,out=5] (G) to[in=-130,out=65] (H) to[out=70,in=-170] (I);
       \node[above right] at (I) {\color{blue!85!black}$\Gamma_h$};
       \node[right] at (D) {\color{black}$\OGamma_{\Theta}$};
       \coordinate (J) at (0.59,0.23);
       \coordinate (K) at (0.425,0.335);
       \draw[blue!85!black,line width=0.75pt,->] (J) -- (K);
       \node[above right,scale=0.7] at (K) {\!\!\color{blue!85!black}$n_h$};
     \end{tikzpicture}
   };
   \draw[line width=0.75pt,->] (Olin.east) to[in=-160,out=-20] node[below] {$\Theta_h$} (Oh.west);
   \draw[line width=0.75pt,->] (Olin.north) to[in=160,out=45] node[above left] {$\Psi$} (Oexa.west);
   \draw[line width=0.75pt,->] (Oexa.east) to[in=130,out=0] node[above right] {$F_h$} (Oh.north);

   \end{tikzpicture}
   \end{center}
   \vspace*{-0.4cm}
   \caption{Sketch of different transformations. $\Psi$ maps the interface approximation $\Gammalin$ onto the exact interface $\Gamma$; $\Theta_h$ is the discrete approximation of $\Psi$. The transformation $F_h := \Theta_h \circ \Psi^{-1}$ has the property $F_h(\Gamma)=\Gamma_h$.}
   \label{fig:trafos}
 \end{figure}

\begin{remark} \label{rem:polext} \rm
The polynomial extension $\mathcal{E}_T$ used in \eqref{eq:psihmap} ensures that the computation of $d_h|_T$ only depends on $\phi_h|_T$, i.e. element-local quantities. This avoids searches in a neighborhood of the element, which enhances the computational efficiency, especially in case of a parallel implementation.
\end{remark}

A key result of the error analysis in \cite{CLARH1} is summarized in the following lemma.
\begin{lemma} \label{lem:phih} The following estimates hold:
\begin{align} 
\Vert \Theta_h - \Psi \Vert_{\infty,\OGamma} + h \Vert D (\Theta_h - \Psi)\Vert_{\infty,\OGamma} & \lesssim h^{k+1},
\label{eq:phih1} \\
\Vert {\rm id} - \Psi \Vert_{\infty,\OGamma} + h \Vert I - D\Psi\Vert_{\infty,\OGamma} & \lesssim h^2. \label{eq:phi2}
\end{align}
\end{lemma}
\begin{proof}
 \cite[Corollary 3.2 and Lemma 3.6]{CLARH1}.
\end{proof}

We emphasize that the constants hidden in the $\lesssim$ notation in \eqref{eq:phih1}, \eqref{eq:phi2}, and also in the estimates in the remainder, do not depend on how $\Gamma^{\rm lin}$ intersects the triangulation $\mT^\Gamma$.
We define the transformed cut mesh domains $\OGamma_{\Theta} := \Theta_h(\OGamma)$, $\OGamma_{\Psi} := \Psi(\OGamma)$, cf. Fig.~\ref{fig:trafos}.
The results in Lemma~\ref{lem:phih} imply that, for $h$ sufficiently small, both $\Theta_h:\, \OGamma \to \OGamma_{\Theta}$ and $\Psi:\, \OGamma \to \OGamma_{\Psi}$ are homeomorphisms.
Furthermore, using \eqref{eq:phih1} one easily derives (\cite[Lemma 3.7]{CLARH1}):
\begin{equation} \label{resdist}
 {\rm dist}(\Gamma_h,\Gamma) \lesssim h^{k+1}.
\end{equation}
For the analysis we also need a result on the approximation of normals in a neighborhood of $\Gamma$. Let $n(x)$, $x \in \Gamma$ be the unit normal to $\Gamma$ (in the direction of $\phi > 0$). In a (sufficiently small) neighborhood of $\Gamma$ we define $n(x):= \nabla \phi / \Vert \nabla \phi \Vert_2$. In case that $\phi$ is a signed distance function this coincides with $n(x) = n(p(x))$ where $p$ is the closest point projection on $\Gamma$.
% Note that $n(x)$ is the unit normal on the isosurface of $\phi$ that contains $x$.
In the following lemma we consider a computable accurate approximation of $n(x)$.
\begin{lemma} \label{Lemnormals}
 For $x \in T \in \TGamma$ define 
\[ \nlin= \nlin(T):= \frac{\nabla \hat \phi_h(x)}{\|\nabla \hat \phi_h(x)\|_2} = \frac{\nabla \hat\phi_{h|T}}{\|\nabla \hat \phi_{h|T}\|_2}, \quad n_h(\Theta_h(x)):= \frac{D\Theta_h(x)^{-T}\nlin}{\|D\Theta_h(x)^{-T}\nlin\|_2}.
\] 
Let $n_{\Gamma_h}(x)$, $x\in \Gamma_h$ a.e., be the unit normal on $\Gamma_h$ (in the direction of $\phi_h > 0$). The following holds
\begin{align}
 \|n_h - n\|_{\infty, \OGamma_{\Theta}} & \lesssim h^k \label{normalres1} \\
 \|n_{\Gamma_h} - n\|_{\infty, \Gamma_h} & \lesssim h^k. \label{normalres2} 
\end{align}
\end{lemma}
\begin{proof} 
 Define the isosurface $\Gammalin_c:= \{\, x \in \OGamma~|~\hat \phi_h(x)=c \,\}$ (not necessarily connected) and its image $\Gamma_{h,c}:= \{\, \Theta_h(x)~|~x \in \Gammalin_c\,\}$. Note that $\Gamma_h=\Gamma_{h,0}$. Take $x \in T \in \TGamma$ and $c$ such that $x \in \Gammalin_c$. The unit normal on $\Gammalin_c$ at $x$ is given by $\nlin$. The unit normal on $\Gamma_{h,c}$ at $\Theta_h(x)$ is given by $n_h(\Theta_h(x))$. Hence, for $y=\Theta_h(x) \in \Gamma_h$ we have $n_{\Gamma_h}(y) - n(y)= n_h(\Theta_h(x))-n(\Theta_h(x))$, and thus \eqref{normalres2} follows from \eqref{normalres1}. Let $\Gamma_c:=\{ x \in \OGamma_{\Psi}~|~ \phi(x)=c \,\}$ be the $c$-isosurface of $\Gamma$. The definition of $\Psi$ implies that
$\Gamma_c= \{\,\Psi(x)~|~x \in \Gammalin_c\,\}$. Thus we get 
\[ n(\Psi(x))=\frac{D\Psi(x)^{-T}\nlin}{\|D\Psi(x)^{-T}\nlin\|_2}. \]
Using this and the result in Lemma~\ref{lem:phih} we get (uniformly in $x$ and $T$):
\begin{align*}
 \|n_h(\Theta_h(x)) - n(\Theta_h(x))\|_2 & \leq \|n_h(\Theta_h(x)) - n(\Psi(x))\|_2 +\|n(\Theta_h(x))- n(\Psi(x))\|_2 \\
 & \lesssim \left\| \frac{D\Theta_h(x)^{-T}\nlin}{\|D\Theta_h(x)^{-T}\nlin\|_2} - \frac{D\Psi(x)^{-T}\nlin}{\|D\Psi(x)^{-T}\nlin\|_2} \right\|_2 + h^{k+1} \\
 & \lesssim\frac{ \Vert \big(D\Theta_h(x)^{-T}- D\Psi(x)^{-T} \big)\nlin \Vert_2}{\|D\Psi(x)^{-T}\nlin\|_2} + h^{k+1} \lesssim h^k.
\end{align*}
In the last inequality we used \eqref{eq:phi2} and \eqref{eq:phih1}. This proves \eqref{normalres1}.
\end{proof}

One further property that we need in the analysis is the uniform local regularity of the mapping $\Theta_h$ that we will show in Lemma \ref{lemregu}. As a preliminary we give the following lemma.

\begin{lemma}\label{lem:dh}
For $h$ sufficiently small, $T\in\mT^\Gamma$, and $F\in\mathcal{F}^\Gamma$, the functions $d_h$ and $\Psi_h$ defined in \eqref{eq:psihmap} and \eqref{eq:psih} have the properties
\begin{subequations}
\begin{align}
 \max_{T\in\mTGamma} \Vert d_h \Vert_{H^{l,\infty}(T)} \lesssim 1, ~ & ~ \max_{T\in\mTGamma} \Vert \Psi_h \Vert_{H^{l,\infty}(T)} \lesssim 1, \qquad l \leq k+1,
\label{eq:psihbound}\\
 \max_{F \in \mFGamma} \Vert \jump{d_h}_F \Vert_{\infty,F} \lesssim h^{k+1}, ~ & ~
 \max_{F \in \mFGamma} \Vert \jump{\Psi_h}_F \Vert_{\infty,F} \lesssim h^{k+1},
\label{eq:jumpbound}
\end{align}
where $\jump{\cdot}_F$ denotes the usual jump operator across the facet $F$.
\end{subequations}
\end{lemma}
\begin{proof}
Similar results are derived in \cite{CLARH1}, e.g.  \cite[Lemma 3.4]{CLARH1}. We include  a proof in Appendix \ref{sec:proof:lem:dh}.
\end{proof}
\ \\
Note that the constants hidden in $\lesssim$ in Lemma~\ref{lem:dh} depend on higher derivatives of the level set function $\phi$ (and thus on the smoothness of $\Gamma$), for example, the constants in \eqref{eq:psihbound} depend on $\|\phi\|_{H^{k+2,\infty}(\OGamma)}$. 
\begin{lemma}\label{lemregu}
  The following holds: $\max_{T\in\mTGamma} \Vert \Theta_h \Vert_{H^{l,\infty}(T)} \lesssim 1$, $l \leq k+1.$
\end{lemma}
\begin{proof}
Recall that $\Theta_h = P_h \Psi_h$, cf.~\eqref{defThetah}. We fix an element $T \in \mTGamma$ and set $\Psi_T = \Psi_h|_T \in C^{k+1}(T)$.
We have 
\[
\Vert \Theta_h \Vert_{H^{l,\infty}(T)} \leq \Vert \Psi_T \Vert_{H^{l,\infty}(T)} + \Vert P_h \Psi_h - \Pi \Psi_T\Vert_{H^{l,\infty}(T)} + \Vert \Pi \Psi_T - \Psi_T\Vert_{H^{l,\infty}(T)},
\]
where $\Pi$ is the nodal interpolation operator into $\mP(T)$. 
For the latter interpolation error we have
\[
\Vert \Pi \Psi_T - \Psi_T \Vert_{H^{l,\infty}(T)} \lesssim \Vert \Psi_T \Vert_{H^{l,\infty}(T)}.
\]
With Lemma \ref{lem:dh} we have
$
\Vert \Psi_T \Vert_{H^{l,\infty}(T)} \lesssim 1
$ 
uniformly in $T$ and hence can bound the first and the last term with $\mathcal{O}(1)$.
It remains to show the estimate for $P_h \Psi_h - \Pi \Psi_T$. Let $\{\psi_i\}_{i\in \mI_T}$, be the nodal basis in $\mP(T)$, as also used in the definition of $P_h$, cf. \eqref{defP}, and $\{x_i\}_{i\in \mI_T}$ the corresponding nodes. We can write (on $T$)
\[
P_h \Psi_h - \Pi \Psi_T = \sum_{i\in \mI_T} (A_{x_i}(\Psi_h) - \Psi_T(x_i)) \psi_i.
\]
For finite element nodes which lie inside an element $T$, i.e. $x_i \not\in \partial T$, we have
$A_{x_i}(\Psi_h) = \Psi_T(x_i)$. For $x_i \in\partial T$ we  use the definition of $A_{x_i}$ and Lemma \ref{lem:dh} and thus obtain: 
\begin{equation} \label{eq:boundq}
|A_{x_i}(\Psi_h) - \Psi_T(x_i)| = |A_{x_i}(\Psi_h - \mathcal{E}_T\Psi_T)| \lesssim \sum_{F \in \mF^{\Gamma}\cap T} \Vert \jump{\Psi_h} \Vert_{\infty,F} \lesssim h^{k+1}.
\end{equation}
In this estimate we used that the number of facets that share a point is uniformly bounded on shape regular meshes. 
With the bound in \eqref{eq:boundq} we get 
\[
\Vert P_h \Psi_h - \Pi \Psi_T \Vert_{H^{l,\infty}(T)}
\lesssim \sum_{i\in \mI_T} \underbrace{|A_{x_i}(\Psi_h) - \Psi_T(x_i)|}_{\lesssim ~ h^{k+1}} \underbrace{\Vert \psi_i \Vert_{H^{l,\infty}(T)}}_{\lesssim ~ h^{-l}} \lesssim 1, \quad l \leq k+1.
\]
which completes the proof.
\end{proof}

We note that Lemma \ref{lemregu} implies that for $u \in H^l(T)$, $T \in \mTGamma$, $l \leq k+1$, we have $\Vert u \circ \Theta_h^{-1} \Vert_{H^l(\Theta_h(T))} \lesssim \Vert u \Vert_{H^l(T)}$.
\section{The isoparametric trace FEM} \label{sec:itracefem}
%{\bf AR: I changed the order: first isoparametric mapping and then the trace operator; I think that this is conceptually better}\\
We start by introducing the space used in the isoparametric trace FEM. We consider the local volume triangulation $\TGamma$ and the standard affine polynomial finite element space $V_h^k$ restricted to $\TGamma$, i.e., $(V_h^k)_{|\OGamma}$. To this space we apply the transformation $\Theta_h$ resulting in the isoparametric space
\begin{equation} \label{isoFEspace} \begin{split}
 V_{h,\Theta}^k & := \{\, v_h \circ \Theta_h^{-1} \mid v_h \in (V_h^k)_{|\OGamma}\, \},\\
 V_{h,\Theta}^{k,0} & := \{ v_h \in V_{h,\Theta}^k \mid \int_{\Gamma_h} v_h \d{s_h}=0 \}.
 \end{split}
\end{equation}
The latter space will be used in our finite element discretization \eqref{eq:varform} below. 
In the error analysis of the method we also use the following larger (infinite dimensional) space:
\[
 \Vregh := \{ v \in H^1(\OGamma_{\Theta})~|~ \tr|_{\Gamma_h} v \in H^1(\Gamma_h)\} \supset V_{h,\Theta}^k,
\]
on which the bilinear forms introduced below are well-defined.
Besides the bilinear form related to the Laplace--Beltrami operator, we also use a \emph{stabilization} $s_h(\cdot, \cdot)$ which we assume to be \emph{symmetric positive semi-definite} and well-defined on the space $\Vregh$. We allow $s_h(\cdot,\cdot) \equiv 0$. The error analysis will reveal that for $s_h(\cdot,\cdot) \equiv 0$ we have optimal order discretization error bounds. For $s_h(\cdot,\cdot) \equiv 0$, however, the stiffness matrix can be very ill-conditioned, depending on how the interface cuts the outer triangulation. The stabilization is used to obtain the usual $\mathcal{O}(h^{-2})$-bound for the condition number of the stiffness matrix, uniformly w.r.t. the cut geometry. In the analysis below we derive conditions on $s_h(\cdot,\cdot)$ such that the latter property holds and we still have optimal order discretization error bounds. Specific choices for $s_h(\cdot,\cdot)$ are discussed in Section \ref{sec:condition-number}.
We introduce the bilinear form 
\begin{align}
 A_h(u,v) & := a_h(u,v)+ s_h(u,v), \quad u,v \in \Vregh, \label{eq:blf2} \\
 a_h(u,v) & := \int_{\Gamma_h} \nabla_{\Gamma_h} u \cdot \nabla_{\Gamma_h} v \, \d{s_h} \label{eq:blf3}.
\end{align}
For the discrete problem we need a suitable extension of the data $f$ to $\Gamma_h$, which is denoted by $f_h$. Specific choices for $f_h$ are discussed in Remark~\ref{daterror}. 
The discrete problem is as follows: 
Find $u_h \in \Vkz$ such that 
\begin{equation} \label{eq:varform} 
 A_h(u_h,v_h) = \int_{\Gamma_h} f_h v_h \, \d{s_h} \quad \text{ for all } v_h \in \Vkz.
\end{equation}

\begin{remark} \label{remnonunique} \rm Because we take the trace of outer finite element functions on the surface approximation $\Gamma_h$ it is natural to introduce the following trace spaces:
\begin{equation} \label{tracespace}
 \begin{split}
V_{h,\Theta}^\Gamma & := \text{tr}|_{\Gamma_h} (V_{h,\Theta}^k), \\
 V_{h,\Theta}^{\Gamma,0} & := \{\, v_h \in V_{h,\Theta}^\Gamma \mid \int_{\Gamma_h} v_h \d{s_h}=0\,\}.
\end{split} 
\end{equation}
Concerning \eqref{eq:varform}, there is the issue that there may be different $w_h,\tilde w_h\in\Vkz$ with the same trace $v_h^\Gamma \in\Vtracez$. In the case $s_h(\cdot,\cdot) \equiv 0$ only trace values on $\Gamma_h$ are used in \eqref{eq:varform}, an thus  we can replace the trial and test space $ \Vkz$ in \eqref{eq:varform} by $V_{h,\Theta}^{\Gamma,0}$. The latter formulation then has a unique solution $u_h^\Gamma \in \Vtracez$, whereas the one in \eqref{eq:varform} may have more solutions, which however, all have the same trace $u_h^\Gamma$. This non-uniqueness issue is directly related to the fact that the set of traces of the outer finite element nodal basis functions form only a frame (in general not a basis) of the trace space $V_{h,\Theta}^\Gamma $. In some of the stabilization approaches introduced further on, the bilinear form $s_h(u_h,v_h)$ will depend also on function values $u_h(x),v_h(x)$ with $x \in \Theta_h(\Omega^\Gamma) \setminus \Gamma_h$. This is the reason why
  we use $\Vkz$ (instead 
of $\Vtracez$) in \eqref{eq:varform}. Adding an appropriate stabilization term $s_h$ will remove the above-mentioned non-uniqueness issue. 
\end{remark}

\begin{remark}[Implementational aspects]\rm
The integrals in \eqref{eq:varform} can be implemented based on numerical integration rules with respect to $\Gammalin$ and the transformation $\Theta_h$. We illustrate this for the bilinear form $a_h(\cdot,\cdot)$, cf. \eqref{eq:blf3}. With $\tilde{u}_h = u_h \circ \Theta_h,~\tilde{v}_h = v_h \circ \Theta_h \in V_h^k$, there holds
\begin{equation}
\label{eq:implementational-aspects}
a_h(u_h, v_h) = \int_{\Gammalin} \tilde g \, \d{\tilde s_h},\quad \tilde g=\mathcal{J}_\Gamma \cdot P_h (D \Theta_h)^{-T} \nabla \tilde{u}_h \cdot 
\ P_h (D \Theta_h)^{-T} \nabla \tilde{v}_h,
\end{equation}
with $P_h = I - n_h n_h^T$ the tangential projection, $n_h = N / \norm{N}$ the unit-normal on $\Gamma_h$ with $N = (D \Theta_h)^{-T} \hat{n}_h$ where $\hat{n}_h = \nabla \hphi / \Vert\nabla \hphi\Vert$ is the normal with respect to $\Gammalin$, and $\mathcal{J}_\Gamma = \det(D \Phi_h) \cdot \Vert N \Vert$.
The mesh transformation $\Theta_h \in (V_h^k)^3$ is explicitly available in the implementation.
The integrand $\tilde g$ is not a polynomial, but smooth on each $T \in \mT^\Gamma$ and thus also on $T \cap \Gammalin$, $T \in \mT^\Gamma$. This means that we only need an accurate integration with respect to the low order geometry $\Gammalin$. We use a numerically stable quadrature rule of exactness degree $2k - 2$ on each $T \cap \Gammalin$, $T \in \mT^\Gamma$, to approximate the integral on the right-hand side of \eqref{eq:implementational-aspects} in the numerical examples. This is motivated by the optimal discretization error bounds in the $H^1$-norm for standard elliptic problems (with variable coefficients) if quadrature is used, see \cite[Thm. 29.1]{CiarletHandbook}.
%
%As the quantities $\mathcal{J}_\Gamma$, $(D \Theta_h)^{-T}$, and $P_h$ are only small ($\mathcal{O}(h)$) perturbations of $1$, $I$, and $I - \hat{n}_h \hat{n}_h^T$, respectively, the numerical stability of a quadrature rule is quaranteed
%as long as the quadrature rule is stable for the ``reference integral'' $\int_{\Gammalin} \tilde{P} \nabla \tilde{u}_h \cdot \tilde{P} \nabla \tilde{v}_h \d{\tilde s_h}$. For instance, a quadrature rule with exactness degree $2k - 2$ with respect to the ``reference integral'' is sufficient for stability and accuracy and has been applied in the numerical examples.
% {\bf AR: concerning the accuracy of the quadrature: Ciarlet derives that it is sufficient to use a scheme that is exact (on the simplex) of degree $2k-2$ to maintain optimal accuracy $h^k$ in $H^1$-norm for standard elliptic problems, cf \cite{CiarletHandbook} Theorem 29.1. Hence, we also suggest degree of exactness $2k-2$? Has Christoph/J\"org done experiments related to this?}
\end{remark}

\section{Discretization error analysis}
\label{sec:error-analysis}
The discretization error analysis that we present is along similar lines as in most papers on finite elements for surface PDEs. We use a Strang Lemma which bounds the discretization error in terms of an approximation error and a consistency error (due to the geometric error). For bounding these two error terms we use results known from the literature. The only essential difference between the analysis below and the analyses known in the literature is that we allow for a generic stabilization $s_h(\cdot,\cdot)$ and introduce conditions on this bilinear form which are sufficient for deriving optimal order discetization error bounds.

In the analysis we need a sufficiently small tubular neigborhood of $\Gamma$.
Let ${\rm sdist}$ denote the signed distance function to $\Gamma$ and $ U_r=\{x\in\R^3\mid  |{\rm sdist}(x)| \le r\}$. The closest-point projector onto $\Gamma$ is denoted by  $p:\, U_r \to \Gamma$. We assume that $r$ is sufficiently small such that the decomposition
\begin{equation} \label{decompo}
 x=p(x)+{\rm sdist}(x)n(x), \quad x \in U_r
\end{equation}
is unique. We assume that $h$ is sufficiently small such that   $\OGamma_{\Theta} \subset U_r$.
 We define the extension $w^e$ of $w \in H^1(\Gamma)$ by $w^e(x):=w(p(x))$ for all $x \in U_r$. We then have $n\cdot \nabla w^e =0$ on $U_r$. In the error analysis we use the natural (semi-)norms
% for the analysis are
% \begin{equation}
%   \enorm{u}^2 = a(u,u) + s(u,u), \quad u \in \Vreg
% \end{equation}
% and
\begin{equation}\label{eq:def-h-norm}
  \enormh{u}^2 :=\|u\|_a^2 + s_h(u,u), \quad \|u\|_a^2:=a_h(u,u), \quad u \in \Vregh.
\end{equation}

\begin{remark} \rm 
On $\Vtracez$ the semi-norm $\|\cdot\|_a$ defines a norm. This follows from a Poincar{\'e} inequality in $\Vtracez$, cf. \eqref{eq:poincare-gamma-h} below.  This implies that, for a solution $u_h\in\Vkz$ of the discrete problem \eqref{eq:varform}, the trace $u_h|_{\Gamma_h}\in \Vtracez$ is unique. The uniqueness of $u_h\in\Vkz$ depends on the stabilization term and will addressed in Remark~\ref{uniquesol} below.
\end{remark}

The error analysis is based on a Strang Lemma:
\begin{lemma} \label{Strang} Let $u \in H^1_\ast(\Gamma)$ be the unique solution of \eqref{eq:weak-LB} with the extension $u^e \in \Vregh$ and $u_h \in \Vkz$ be a solution of \eqref{eq:varform}. Then we have the discretization error bound
\begin{equation} \label{Strangbound}
 \enormh{u^e-u_h} \leq 2 \min_{v_h \in \Vkz} \enormh{u^e - v_h} + \sup_{w_h \in \Vkz} \frac{|A_h(u^e,w_h)-\int_{\Gamma_h} f_h w_h \, \d{s_h}|}{\enormh{w_h}}.
\end{equation}
\end{lemma}
\begin{proof}
For $v_h \in \Vkz$ and $w_h = u_h - v_h \in \Vkz$ we have 
\[
 \Vert u_h - v_h \Vert_h \leq \frac{A_h(u_h-v_h,u_h-v_h)}{\Vert w_h \Vert_h} \leq \Vert u^e - v_h \Vert_h + \frac{A_h(u_h-u^e,w_h)}{\Vert w_h \Vert_h}.
\]
Together with \eqref{eq:varform} and the triangle inequality
$
\enormh{u^e - u_h} \leq \enormh{u^e - v_h} + \enormh{u_h - v_h}
$ the claim follows.
\end{proof}

In the following two subsections we analyze the terms in the Strang error bound.

\subsection{Approximation error}
We first recall some known approximation results from the literature. The isoparametric interpolation $\ITk:\, C(\OGamma_\Theta) \to \Vk$ is defined by $(\ITk v)\circ \Theta_h= \Ik(v\circ \Theta_h)$. Using the property in Lemma~\ref{lemregu}, the theory on isoparametric finite elements, cf.~\cite{Lenoir86},
%,GiraultRaviart}, {\bf AR:check} \todo[inline]{JG: I verified it with Lenoir86; is there anything in GiraultRaviart on curved FE?}
yields the following optimal interpolation error bound for $0 \leq l \leq k+1$:
\begin{equation} \label{interpol}
 \|v - \ITk v\|_{H^l(\Theta_h(T))} \lesssim h^{k+1-l} \|v\|_{H^{k+1}(\Theta_h(T))} ~~ \text{for all}~v \in H^{k+1}(\Theta_h(T)), ~T \in \mT.
\end{equation}
We will also need the following trace estimate \cite{hansbo2002unfitted}:
\begin{equation} \label{Hansboest}
\|v\|_{L^2(\Gamma_T)}^2 \lesssim h^{-1} \|v\|_{L^2(\Theta_h(T))}^2 + h \|\nabla v\|_{L^2(\Theta_h(T))}^2, \quad v \in H^1(\Theta_h(T)),
\end{equation}
with $\Gamma_T:=\Gamma_h \cap \Theta_h(T)$.
To obtain an interpolation in $\Vkz$, we define 
\[
 \ITkz v := \ITk v - |\Gamma_h|^{-1} \int_{\Gamma_h} \ITk v \, \d{s_h}.
\]
For this interpolation operator we have the following error estimate.
\begin{lemma} \label{lemintb} The following holds for all $v \in H^{k+1}(\OGamma_\Theta))$, $l=0,1,2$:
\[
  \| v - \ITkz v \|_{H^l(\OGamma_\Theta)} \lesssim h^{k+1-l} \|v\|_{H^{k+1}(\OGamma_\Theta)} + h^\frac12 \left|\int_{\Gamma_h} v \d{s_h} \right|.
\]
\end{lemma}
\begin{proof}
From a triangle inequality and $|\OGamma_\Theta| \lesssim | \Gamma_h | h$ we get: 
\[
  \| v - \ITkz v \|_{H^l(\OGamma_\Theta)} \lesssim \| v - \ITk v \|_{H^l(\OGamma_\Theta)} + h^\frac12 \left\vert \int_{\Gamma_h} \ITk v \d{s_h} \right\vert, \quad l=0,1,2.
\]
The first term on the right-hand side can be bounded by $ch^{k+1-l} \|v\|_{H^{k+1}(\OGamma_\Theta)}$ using the result in \eqref{interpol}.
For the second term we have, using \eqref{Hansboest}, 
\begin{align*}
 h^\frac12 & \left\vert \int_{\Gamma_h} \ITk v \d{s_h} \right\vert 
\leq 
h ^\frac12 \left|\int_{\Gamma_h} v - \ITk v \d{s_h} \right| + h^\frac12 \left|\int_{\Gamma_h} v \d{s_h} \right| \lesssim h^\frac12 \| v - \ITk v \|_{L^2(\Gamma_h)} \\
 & + h^\frac12 \left|\int_{\Gamma_h} v \d{s_h} \right| \lesssim \| v - \ITk v \|_{L^2(\OGamma_\Theta)} + h \| v - \ITk v \|_{H^1(\OGamma_\Theta)}+ h^\frac12 \left|\int_{\Gamma_h} v \d{s_h} \right|.% \\
% & \lesssim h^{k+1} \|v\|_{H^{k+1}(\OGamma_\Theta)} + h^\frac12 \left|\int_{\Gamma_h} v \d{s_h} \right|.
\end{align*}
Together with \eqref{interpol} this completes the proof.
\end{proof}
\ \\[1ex]
\begin{lemma} \label{intertrace} For the space $\Vkz$ we have the approximation error estimate
\begin{equation}
 \begin{split} 
 & \min_{v_h \in \Vkz} \big( \|v^e-v_h\|_{L^2(\Gamma_h)} + h \|\nabla(v^e-v_h)\|_{L^2(\Gamma_h)} \big) \\
 & \leq \|v^e-\ITkz v^e \|_{L^2(\Gamma_h)} + h \|\nabla(v^e- \ITkz v^e)\|_{L^2(\Gamma_h)} 
 \lesssim h^{k+1} \|v\|_{H^{k+1}(\Gamma)} 
\end{split}
\end{equation}
for all $v \in H^{k+1}(\Gamma) \cap H^1_\ast(\Gamma)$. (Recall that $v^e$ is a normal extension of $v$.)
\end{lemma}
\begin{proof} Take $v \in H^{k+1}(\Gamma) \cap H^1_\ast(\Gamma)$, hence $\int_{\Gamma} v \d s =0 $ holds.
From \eqref{Hansboest} and Lemma~\ref{lemintb} we obtain
\begin{align}
  & \|v^e-\ITkz v^e \|_{L^2(\Gamma_h)} + h \|\nabla(v^e- \ITkz v^e)\|_{L^2(\Gamma_h)} \nonumber \\ & \lesssim h^{-\frac12}\| v^e - \ITkz v^e \|_{L^2(\OGamma_\Theta)} + h^\frac12 \| v^e - \ITkz v^e \|_{H^1(\OGamma_\Theta)} + h^\frac32 \sum_{T \in \mT^\Gamma} \| v^e - \ITkz v^e \|_{H^2(\Theta_h(T))} \nonumber \\
   & \lesssim h^{k+\frac12} \|v^e \|_{H^{k+1}(\OGamma_\Theta)}+ \left|\int_{\Gamma_h} v^e \d{s_h} \right|. \label{tt}
\end{align}
Now note that 
\begin{equation} \label{nb}
 \|D^\mu u^e\|_{L^2(\OGamma_\Theta)} \lesssim h^\frac12 \|u\|_{H^m(\Gamma)} \quad \text{for all}~~u \in H^m(\Gamma), ~|\mu| \leq m,
 \end{equation} 
holds, cf. \cite[Lemma 3.1]{reusken2015}. Using this we get
\begin{equation} \label{pp}
  \|v^e \|_{H^{k+1}(\OGamma_\Theta)} \lesssim h^\frac12 \|v \|_{H^{k+1}(\Gamma)}. 
\end{equation}
We now treat the term $\lvert \int_{\Gamma_h} v^e \d{s_h}\rvert$ in \eqref{tt}.
Recall that $p$ is the closest point projection on $\Gamma$. We use standard results from the literature, e.g. \cite{Demlow06}. For the transformation of the surface measure the relation 
\begin{equation} \label{surfmeasuer}
 \mu_h ds_h(x) = ds(p(x)), \quad \text{for}~~ x \in \Gamma_h,
\end{equation}
holds \cite[Proposition A.1]{Demlow06}. The function $\mu_h$ satisfies
\begin{equation} \label{estmu}
 \|1- \mu_h\|_{\infty,\Gamma_h} \lesssim h^{k+1},
\end{equation}
cf. \cite{Demlow06,reusken2015}.
Using this and $v \in H_\ast^1(\Gamma)$ we get
\begin{equation} \label{ttp} \begin{split}
 \left| \int_{\Gamma_h} v^e \d{s_h} \right|& = 
\Big| \int_{\Gamma_h} v^e \d{s_h} - \int_{\Gamma} v \d{s} \Big|
 = \Big| \int_{\Gamma_h} v^e ( 1 - \mu_h) \d{s_h} \Big| \\ & \lesssim h^{k+1} \Vert v^e \Vert_{L^2(\Gamma_h)} \lesssim h^{k+1} \Vert v \Vert_{L^2(\Gamma)}. 
\end{split} \end{equation}
Combining this with the results in \eqref{tt}, \eqref{pp} completes the proof.
\end{proof}

Using this interpolation bound one easily obtains a bound for the approximation term in the Strang Lemma.
\begin{lemma} Assume that the stabilization satisfies
\begin{equation} \label{conds1}
 s_h(w,w) \lesssim h^{-3}\|w\|_{L^2(\OGamma_\Theta)}^2 + h^{-1} \|\nabla w\|_{L^2(\OGamma_\Theta)}^2 \quad \text{for all} ~~w \in \Vregh.
\end{equation}
Then 
\[
  \min_{v_h \in \Vkz} \enormh{u^e - v_h} \lesssim h^k \|u\|_{H^{k+1}(\Gamma)} \quad \text{holds for all}~~u \in H^{k+1}(\Gamma) \cap H_\ast^1(\Omega).
\]
\end{lemma}
\begin{proof} 
 Take $u \in H^{k+1}(\Gamma)\cap H_\ast^1(\Gamma)$ and $v_h:= \ITkz u^e$. From Lemma~\ref{intertrace} we get $\|u^e-v_h\|_a \lesssim h^k \|u\|_{H^{k+1}(\Gamma)} $.
 From the assumption \eqref{conds1} combined with the results in
  Lemma~\ref{lemintb}
and the estimates \eqref{pp} and \eqref{ttp} we get $s_h(u^e-v_h,u^e-v_h)^\frac12 \lesssim h^k \|u\|_{H^{k+1}(\Gamma)}$, which completes the proof.
\end{proof}

\subsection{Consistency error}
We derive a bound for the consistency error term on the right-hand side in the Strang estimate \eqref{Strangbound}. We have to quantify the accuracy of the data extension $f_h$. 
We recall the definition of $\mu_h$, cf. \eqref{surfmeasuer}, and define
\[
  \delta_f:=f_h- \mu_h f^e \quad \text{on}~~\Gamma_h.
\]
 \begin{lemma} \label{lem:conserr}
Let $u \in H_\ast^1(\Gamma)$ be the solution of \eqref{eq:weak-LB}. Assume that the data error satisfies $\|\delta_f\|_{L^2(\Gamma_h)} \lesssim h^{k+1} \|f\|_{L^2(\Gamma)}$ and the stabilization satisfies
\begin{equation} \label{conds2}
 \sup_{w_h \in \Vkz} \frac{s_h(u^e,w_h)}{\enormh{w_h}} \lesssim h^{l}\|f\|_{L^2(\Gamma)}, \quad \text{with $l=k$ or $l=k+1$}.
\end{equation}
 Then the following holds:
\[ \sup_{w_h \in \Vkz} \frac{|A_h(u^e,w_h)-\int_{\Gamma_h} f_h w_h \, \d{s_h}|}{\enormh{w_h}} \lesssim h^{l}\|f\|_{L^2(\Gamma)}.
\]
\end{lemma}
\begin{proof}
 We use the splitting
\[
   |A_h(u^e,w_h)-\int_{\Gamma_h} f_h w_h \, \d{s_h}| \leq |a_h(u^e,w_h)-\int_{\Gamma_h} f_h w_h \d{s_h}| + s_h(u^e,w_h).
\]
The first term has been analyzed in \cite{reusken2015}, Lemma 5.5. In the analysis one essentially only needs the closeness properties in \eqref{resdist}, \eqref{normalres2} and the bound on the data error.
The analysis yields
\[
  \sup_{w_h \in \Vkz} \frac{|a_h(u^e,w_h)-\int_{\Gamma_h} f_h w_h \, \d{s_h}|}{\enormh{w_h}} \lesssim h^{k+1}\|f\|_{L^2(\Gamma)}.
\]
We use assumption \eqref{conds2} to bound the second term.
\end{proof}

\subsection{Optimal $H^1$-error bound}
As an immediate consequence of the previous results we obtain the following main theorem.
\begin{theorem} \label{mainthm}
Let $u \in H^{k+1}(\Gamma)\cap H_\ast^1(\Gamma)$ be the solution of \eqref{eq:weak-LB} and $u_h \in \Vkz$ a solution of \eqref{eq:varform}. Assume that the data error satisfies $\|\delta_f\|_{L^2(\Gamma_h)} \lesssim h^{k+1} \|f\|_{L^2(\Gamma)}$ and the stabilization satisfies the conditions \eqref{conds1}, \eqref{conds2}. Then the following holds:
\begin{equation}
  \enormh{u^e - u_h } \lesssim h^k \Vert u \Vert_{H^{k+1}(\Gamma)} + h^{l} \Vert f \Vert_{L^2(\Gamma)}.
\end{equation}
\end{theorem}
\begin{remark} \rm
\label{daterror}
 We comment on the data error $\|\delta_f\|_{L^2(\Gamma_h)}$, with $\delta_f=f_h -\mu_h f^e$. For the choice $f_h= f^e- \frac{1}{|\Gamma_h|} \int_{\Gamma_h} f^e \d{s_h}$, which in practice often can \emph{not} be realized, we obtain, using \eqref{estmu}, the data error bound $\| \delta_f\|_{L^2(\Gamma_h)} \leq c h^{k+1} \|f\|_{L^2(\Gamma)}$ (as in Lemma \ref{intertrace}). For this data error bound we only need $f \in L^2(\Gamma)$, i.e., we avoid higher order regularity assumptions on $f$. Another, more feasible, possibility arises if we assume $f$ to be defined in a neighborhood $U_{\delta_0}$ of $\Gamma$. As extension one can then use
\begin{equation} \label{extens}
 f_h(x)= f(x)- c_f, \quad c_f:=\frac{1}{|\Gamma_h|} \int_{\Gamma_h} f\, \d{s_h}.
\end{equation}
Using $\int_{\Gamma} f \, \d{s}=0$, \eqref{resdist}, \eqref{estmu} and a Taylor expansion we get $|c_f| \leq c h^{k+1} \|f\|_{H^{1,\infty}(U_{\delta_0})}$ and $\|f- \mu_h f^e\|_{L^2(\Gamma_h)} \leq c h^{k+1} \|f\|_{H^{1,\infty}(U_{\delta_0})}$. Hence, we obtain a data error bound $\|\delta_f\|_{L^2(\Gamma_h)} \leq \hat c h^{k+1}\|f\|_{L^2(\Gamma)}$ with $\hat c=\hat c(f)=c \|f\|_{H^{1,\infty}(U_{\delta_0})} \|f\|_{L^2(\Gamma)}^{-1}$ and a constant $c$ independent of $f$. Thus in problems with smooth data, $ f \in H^{1,\infty}(U_{\delta_0})$, the extension defined in \eqref{extens} satisfies the condition on the data error in Theorem~\ref{mainthm}.
\end{remark}

\begin{corollary} As a trivial consequence of the theorem above we obtain optimal $H^1$-error bounds for the case \emph{without} stabilization, i.e., $s_h(\cdot,\cdot)\equiv 0$. 
\end{corollary}
% \section{Condition number bound}
\section{Condition number analysis}
\label{sec:condition-number}
% In this section, we derive a bound for the condition number of the stiffness matrix of \eqref{eq:varform}.
In this section, we derive condition number bounds for the stiffness matrix resulting from the discretization \eqref{eq:varform}.
It is well-known that in the case $s_h(\cdot,\cdot)=0$ already for $k=1$ the stiffness matrix of the discrete problem may have a condition number that does not scale like $h^{-2}$. This is due to the fact that the condition number depends on the position of the interface with respect to the volume triangulation. Remedies were proposed in \cite{BHL15,burman16fullgradcmame,reusken2015} for the case $k=1$. Below we formulate an assumption on the generic stabilization $s_h(\cdot,\cdot)$ that, together with \eqref{conds1} and \eqref{conds2}, is sufficient to guarantee a stiffness matrix condition number of $\mathcal{O}(h^{-2})$, while still preserving optimal order a-priori discretization error bounds. We thus have a general framework for comparing and analyzing different stabilization techniques, similar to the approach used in \cite{burmanembedded}. In Sections \ref{sec:ghostpen}--\ref{sec:fullgradvol}, for $k=1$
we discuss three stabilizations known from the literature. In Section \ref{sec:normal-derivative-volume}, we treat a fourth stabilization, cf. also \cite{burmanembedded}, which is easy to implement and satisfies the aforementioned conditions also in the higher order case $k\geq 1$.

% {\bf AR: the presentation below can maybe be improved; maybe I missed details related to the kernel of $\Delta_\Gamma$}\\
Let $\bu \in \mathbb{R}^N$ be the representation of $u_h \in V_{h,\Theta}^k$ with respect to the standard nodal basis in $V_{h,\Theta}^k$, i.e., $u_h =\sum_{i=1}^N u_i \phi_i$, and similarly $\bv \in \mathbb{R}^N$ is the representation of $v_h \in V_{h,\Theta}^k$.
The \emph{volume} mass matrix is defined by 
\[
  \langle \bM \bu,\bv\rangle= \int_{\OGamma_\Theta} u_h v_h \, \d{x} \quad \text{for all}~~u_h,v_h \in \Vk.
\]
This matrix is symmetric positive definite and from standard finite element theory it follows that there are positive constants $c_L$ and $c_U$, depending only on $k$ and on the shape regularity of the outer triangulation $\mT$, such that
\begin{equation} \label{estM}
 c_L \leq \frac{\langle \bM \bu, \bu \rangle}{\langle\bu,\bu\rangle} \leq c_U \quad \text{for all}~~\bu \in \mathbb{R}^N, ~ \bu \neq 0.
\end{equation}
The stiffness matrix $\bS \in \mathbb{R}^{N\times N}$ is defined by
\[
  \langle \bS \bu, \bv\rangle= A_h(u_h,v_h) \quad \text{for all}~~u_h,v_h \in V_{h,\Theta}^k.
\]
This matrix is symmetric positive semi-definite. In the discretization we search for $u_h \in V_{h,\Theta}^k$ with $\int_{\Gamma_h} u_h \d s_h=0$. For the vector representation of the latter constraint we introduce $\bc \in \mathbb{R}^N$ with $c_i:= \int_{\Gamma_h} \phi_i \d s_h$, $1\leq i \leq N$, and define 
\begin{equation} \label{cdef}
\mathbb{R}^N_\ast:= \{\, \bu \in \mathbb{R}^N~|~\bu\cdot \bc =0\,\}. 
\end{equation}
Hence $u_h \in \Vkz$ iff $\bu \in \mathbb{R}^N_\ast$.
Let $q_L >0$, $q_U >0$ be such that
\begin{equation}\label{estA}
 q_L \leq \frac{A_h(u_h,u_h)}{\|u_h\|_{L^2(\OGamma_\Theta)}^2} \leq q_U \quad \text{for all}~~u_h \in \Vkz, ~u_h \neq 0.
\end{equation}
The estimates in \eqref{estM} and \eqref{estA} imply
\begin{equation} \label{estS}
 \frac{\max_{\bu \in \mathbb{R}^N_\ast, \|\bu\|_2=1}\langle \bS \bu, \bu \rangle}{\min_{\bu \in \mathbb{R}^N_\ast, \|\bu\|_2=1}\langle \bS \bu, \bu \rangle} =: \operatorname{cond_\ast}(\bS) \leq \frac{c_U q_U}{c_Lq_L}.
\end{equation}
 Hence, we want to obtain (sharp) estimates for the bounds in \eqref{estA}. We are interested in the dependence of $q_L,q_U$ on $h$. Recall that in the inequalities $\lesssim$ (also used below) the constant is independent of $h$ \emph{and of how the surface cuts the volume triangulation}. Concerning the upper bound in \eqref{estA} we have the following result.
\begin{lemma} \label{lemqU}
 Assume that the stabilization satisfies \eqref{conds1}. The following holds:
\begin{equation} \label{b1}
 \frac{A_h(u_h,u_h)}{\|u_h\|_{L^2(\OGamma_\Theta)}^2} \lesssim h^{-3} \quad \text{for all}~~u_h \in \Vk, ~u_h \neq 0.
\end{equation}
\end{lemma}
\begin{proof}
We use Lemma \ref{lem:phih} and finite element inverse inequalities which we apply to $\hat{u}_h := u_h \circ \Theta_h |_T$, $T \in \mT^\Gamma$, so that for all $u_h \in \Vk$ there holds
\begin{align*}
  \Vert \nabla u_h \Vert_{\Gamma_h \cap \Theta_h(T)}^2 \!
\lesssim \!
  \Vert \nabla \hat{u}_h \Vert_{\Gammalin \cap T}^2 \!
\lesssim \!
h^{-1} \Vert \nabla \hat{u}_h \Vert_{L^2(T)}^2 \!
\lesssim \!
h^{-3} \Vert \hat{u}_h \Vert_{L^2(T)}^2 \!
\lesssim \!
h^{-3} \Vert u_h \Vert_{L^2(\Theta_h(T))}^2.
\end{align*}
% We use the estimate \eqref{Hansboest} and apply it to $v=\nabla u_h|_{\Theta_h(T)}$, $u_h \in \Vk$:
% \begin{align*}
%   \Vert \nabla u_h \Vert_{\Gamma_h \cap \Theta_h(T)}^2 &
%   \lesssim h^{-1}\Vert \nabla u_h \Vert_{L^2(\Theta_h(T))}^2+ h \Vert \nabla u_h \Vert_{H^1(\Theta_h(T))}^2 \\ & \lesssim h^{-3} \Vert u_h \Vert_{L^2(\Theta_h(T))}^2\quad\text{for all}~~u_h \in \Vk.
% \end{align*}
% where in the last estimate we used a finite element inverse inequality.
Summing over $T \in \mT^\Gamma$ we get
\begin{align*}
 a_h(u_h,u_h) & = \Vert \nabla_{\Gamma_h} u_h \Vert_{L^2(\Gamma_h)}^2 \leq \Vert \nabla u_h \Vert_{L^2(\Gamma_h)}^2 \lesssim h^{-3} \| u_h\|_{L^2(\OGamma_\Theta)}^2. 
\end{align*}
The assumption \eqref{conds1} and an inverse inequality yield the same bound for $s_h(\cdot,\cdot)$
\[
 s_h(u_h,u_h) \lesssim h^{-3} \|u_h\|_{L^2(\OGamma_\Theta)}^2 + h^{-1} \|\nabla u_h\|_{L^2(\OGamma_\Theta)}^2 \lesssim h^{-3} \|u_h\|_{L^2(\OGamma_\Theta)}^2.
\]
\end{proof}

From Lemma~\ref{lemqU} and the result in \eqref{estS}, we obtain as a corollary the following main result.
\begin{theorem}\label{coro:condition-number}
Assume that the stabilization satisfies \eqref{conds1} and that
\begin{equation}
a_h(u_h,u_h) + s_h(u_h,u_h) \gtrsim h^{-1} \| u_h\|_{L^2(\OGamma_\Theta)}^2 \quad \text{for all}~~u_h \in \Vkz. \label{conds3}
\end{equation}
Then, the spectral condition number satisfies
\begin{equation} \label{condbound}
 \operatorname{cond_\ast}(\bS) \lesssim h^{-2}.
\end{equation}
\end{theorem}

\begin{remark} \label{uniquesol} \rm From the previous theorem it follows that if the stabilization satisfies \eqref{conds1} and \eqref{conds3} then the stiffness matrix is regular and thus the discrete problem \eqref{eq:varform} has a unique solution, cf. Remark~\ref{remnonunique}.
\end{remark}

\subsection{Assumptions on the stabilization term}
We summarize the assumptions on the stabilization term $s_h$ used to derive Theorem \ref{mainthm} (optimal discretization error bound) and Theorem \ref{coro:condition-number} (condition number bound):
\begin{subequations} \label{conds}
\begin{align}
  s_h(w,w) & \lesssim h^{-3}\|w\|_{L^2(\OGamma_\Theta)}^2 + h^{-1} \|\nabla w\|_{L^2(\OGamma_\Theta)}^2 ~~\text{for all}~w \in \Vregh, \label{conds1A} \\ 
\sup_{w_h \in \Vkz} \frac{s_h(u^e,w_h)}{\enormh{w_h}} & \lesssim h^{l}\|f\|_{L^2(\Gamma)}, \quad \text{with }l=k \text{ or } l=k+1, \label{conds2A} \\
\!\!\! a_h(u_h,u_h) + s_h(u_h,u_h) & \gtrsim h^{-1} \| u_h\|_{L^2(\OGamma_\Theta)}^2 \quad \text{for all}~~u_h \in \Vkz. \label{conds3A} 
% \\
% \textrm{\bf will be removed:} s_h(u_h,u_h) & \gtrsim h \|\tilde n \cdot \nabla u_h\|_{L^2(\OGamma_\Theta)}^2 \quad \text{for all}~~u_h \in V_{h,\Theta}^k. \label{oldconds3A}
\end{align}
\end{subequations}
% [with $\tilde n$ in \eqref{conds3A} such that $\|\tilde n - n\|_{\infty,\OGamma_\Theta} \lesssim h$ holds.]
% $\leftarrow$\textrm{\bf will be removed.} 
%
The first two are needed for optimal discretization error bounds, and the first and third one are needed for the uniform $\mathcal{O}(h^{-2})$ condition number bound.
We note that we only need $l=k$ in \eqref{conds2A} to obtain optimal $H^1$ error bounds. Having \eqref{conds2A} with $l=k+1$ may be useful in order to derive $L^2$ error bounds. The latter has not been studied, yet.\\
%{\bf AR: should we briefly discuss the ``nature'' of these conditions?}
% \begin{remark}\label{rem:hk}
% If \eqref{conds2A} is replaced with the weaker assumption
% \begin{equation}\label{conds2Ab}
% \sup_{w_h \in \Vkz} \frac{s_h(u^e,w_h)}{\enormh{w_h}} \lesssim h^{k}\|f\|_{L^2(\Gamma)}, \tag{\ref{conds2A}$^{\ast}$}
% \end{equation}
% the $H^1$-error-bound in Theorem~\ref{mainthm} continues to hold with $h^{k+1} \|f\|_{L^2(\Gamma)}$ replaced by $h^{k} \|f\|_{L^2(\Gamma)}$. Corollary \ref{coro:condition-number} continues to hold without modifications.
% \end{remark}

% \clnote{We should use the weaker condition \eqref{conds2Ab} only. }

\subsection{Ghost penalty stabilization} \label{sec:ghostpen}
The ``ghost penalty'' stabilization is introduced in \cite{Burman2010} as a stabilization mechanism for unfitted finite element discretizations. In \cite{BHL15}, it is applied to a trace finite element discretization of the Laplace--Beltrami equation with piecewise linear finite elements ($k=1$). This stabilization is defined by the facet-based bilinear form 
\[
s_h(u_h,v_h) = \rho_s \sum_{F \in \mathcal{F}_h} \int_F \jump{\nabla u_h \cdot n_h} \jump{\nabla v_h \cdot n_h} \d{s},
\]
with a stabilization parameter $\rho_s > 0$, $\rho_s \simeq 1$, and with $n_h$ the normal to the facet.
For $k=1$, the assumptions in \eqref{conds} are satisfied due to results in \cite{BHL15}: Assumption \eqref{conds1A} follows from \cite[Lemma 4.6]{BHL15}, \eqref{conds2A} follows from $\jump{\nabla u^e \cdot n_h}=0$ for the smooth solution $u$, and \eqref{conds3A} follows from \cite[Lemma 4.5]{BHL15}.

A less nice property of the ghost-penalty method is that the jump of the derivatives on the element-facets changes the sparsity pattern of the stiffness matrix. The facet-based terms enlarge the discretization stencils. 

To our knowledge, there is no higher order version of the ghost penalty method for surface PDEs which provides a uniform bound on the condition number.
% \todo[inline]{I don't think that a higher order version with ghost penalty is working. The ghost penalty will enforce a smooth (extended) function in $\Omega^\Gamma_{\Theta}$, but this still leaves many degrees of freedoms in normal direction. I don't think that one easily gets \eqref{conds3A}.}
% The ghost penalty mechanism has also been considered for higher order discretization, although not specifically for surface PDEs, yet \todo[inline]{Literature!}
% A corresponding higher order version of the stabilization bilinear form would take the form
% \[
% s_h(u_h,v_h) = \sum_{l=1}^k \rho_{s,l} \ h^{2(l-1)} \sum_{F \in \mathcal{F}_h} \int_F \bigjump{\frac{ \partial^l u_h}{\partial n^l}} \bigjump{\frac{ \partial^l v_h}{\partial n^l}} \d{s}
% \]
% with stabilization parameters $\rho_{s,l},~l=1,..,k$.

\subsection{Full gradient surface stabilization} \label{sec:fullgradsurf}
The ``full gradient'' stabilization is a method which does not rely on facet-based terms and keeps the sparsity pattern intact. It was introduced in \cite{DER14,reusken2015}. The bilinear form which describes this stabilization is
\begin{equation} \label{fullgradient}
 s_h(u_h,v_h):= \int_{\Gamma_h} \nabla u_h \cdot n_h \, \nabla v_h \cdot n_h \, \d{s_h},
\end{equation}
where $n_h$ denotes the normal to $\Gamma_h$.
Thus, we get $A_h(u_h,v_h)=\int_{\Gamma_h} \nabla u_h \cdot \nabla v_h \, \d{s_h}$, which explains the name of the method. 
The stabilization is very easy to implement. The conditions \eqref{conds1A} and \eqref{conds2A} hold for any $k$ with $l=k$, cf. \cite[Lemma 5.5]{reusken2015}.

For the case $k=1$, it is shown in \cite{reusken2015} that one has a uniform condition number bound as in \eqref{condbound}. The proof in \cite{reusken2015} relies on estimates similar to \eqref{conds1A} and \eqref{conds3A}, see \cite[Lemma 6.3]{reusken2015}.
%where the $L^2$-norm $\Omega^\Gamma$($\Omega^\Gamma_{\Theta}$) is replaced with a weighted version). 
For the case $k >1$, full gradient stabilization does not result in a uniform bound on the condition number, cf. \cite[Remark 6.5]{reusken2015}. This can be traced back to a failure to satisfy \eqref{conds3A}.
% (or the adapted version from \cite[Lemma 6.3]{reusken2015})

\subsection{Full gradient volume stabilization} \label{sec:fullgradvol}
Another ``full gradient'' stabilization was introduced in \cite{burman16fullgradcmame}. It uses the full gradient in the volume instead of (only) on the surface. The stabilization bilinear form is
\[
 s_h(u_h,v_h) = \rho_s \int_{\OGamma_{\Theta}} \nabla u_h \cdot \nabla v_h \d{x},
\]
with a stabilization parameter $\rho_s > 0$, $\rho_s \simeq h$. Again, it is easy to implement this stabilization as its bilinear form is provided by most finite element codes.

Condition \eqref{conds1A} is satisfied as $s_h(w,w) \simeq h \Vert \nabla w \Vert_{L^2(\Omega_\Theta^\Gamma)}^2$. In \cite[Lemma 4.2]{burman16fullgradcmame}, the condition \eqref{conds3A} is shown to hold. Hence, the bound \eqref{condbound} for the spectral condition number holds for arbitrary $k \geq 1$. 
The consistency condition \eqref{conds2A}, however, is satisfied only in the case $l=k=1$, cf. \cite[Lemma 6.2, Term III]{burman16fullgradcmame}.

\subsection{Normal derivative volume stabilization}
\label{sec:normal-derivative-volume}
In the lowest-order case $k=1$, the stabilization methods discussed in Section \ref{sec:ghostpen}, \ref{sec:fullgradsurf}, and \ref{sec:fullgradvol} satisfy the conditions \eqref{conds1A}, \eqref{conds2A}, and \eqref{conds3A}. For $k>1$, however, for all of these methods at least one of the three conditions in \eqref{conds} is violated. 
We now introduce a stabilization method, also considered in \cite{burmanembedded}, which fulfills \eqref{conds} for arbitrary $k \geq 1$. Its bilinear form is given by
\begin{equation}\label{eq:def-normal-gradient-volume-stab}
 s_h(u_h,v_h):= \rho_s \int_{\OGamma_\Theta} n_h \cdot \nabla u_h \, n_h \cdot \nabla v_h \, \d{x}
\end{equation}
with $n_h$ as in Lemma~\ref{Lemnormals} and $\rho_s>0$.
This is a (natural) variant of the stabilizations treated in Section \ref{sec:fullgradsurf} and \ref{sec:fullgradvol}. As in the full gradient surface stabilization only normal derivatives are added, but this time (as in the full gradient volume stabilization) in the volume $\OGamma_\Theta$. 
The implementation of this stabilization term is fairly simple as it fits well into the structure of many finite element codes.
%works only for trace FEM (we need volume integrals). .. check the conditions above. 
The scaling of the stabilization parameter $\rho_s$ is assumed to satisfy
\begin{equation}\label{eq:as-rho}
  h \lesssim \rho_s \lesssim h^{-1}.
\end{equation}
In the next section we prove that this stabilization satisfies all three conditions in \eqref{conds}, for arbitrary $k \geq 1$.
\begin{remark} \label{comextra} \rm
  This stabilization method has also been introduced in the recent preprint  \cite{burmanembedded}. There it is used in the setting of \emph{linear} trace FEM (or CutFEM) for the discretization  of partial differential equations on embedded manifolds of arbitrary codimension. In that paper an important result \cite[Proposition 8.8]{burmanembedded} is derived that is very similar to a main result in the next section, Lemma~\ref{lemma3A}. The analysis given in  \cite{burmanembedded} applies also to the setting of higher order trace FEM. The analysis given in section~\ref{sec:new-stab} differs from the one given in  \cite{burmanembedded}.
  Due to the isoparametric mapping, our analysis has to consider curved tetrahedra and corresponding isoparametric finite element spaces while straight simplices and piecewise polynomial spaces are considered in \cite{burmanembedded}.
  In \cite{burmanembedded} the analysis is based on the concept of a ``fat intersection covering'', cf. \cite{BHL15}, while we use a more direct approach in our analysis.
\end{remark}

\section{Analysis of the normal derivative volume stabilization} \label{sec:new-stab}
 In this section we analyze the normal derivative volume stabilization \eqref{eq:def-normal-gradient-volume-stab}. We will prove that this method satisfies the conditions in \eqref{conds}. 
 The structure of this section is as follows. In section \ref{sec:proofs:easy} we consider the, relatively easy to prove, conditions \eqref{conds1A} and \eqref{conds2A}. It turns out that condition \eqref{conds3A} is more difficult to prove and requires more analysis, which is given in section~\ref{sectproof3A}.
% {\bf Ar: modify rest of intro}\\
%We first consider the case where $\OGamma_\Psi$ is a tubular neighborhood in section \ref{sec:tubularneighborhood}. In this case the problem is dramatically simplified, but not realistic. A more realistic and comparably feasible case is the case of a $\delta$-acute mesh which is discussed in section \ref{sec:delta-acute}. In section \ref{sec:quasiuniform} the general case of quasi-uniform meshes is considered. 
%Using the insight gained from the sections \ref{sec:tubularneighborhood} and \ref{sec:delta-acute} the controllability of discrete functions on a specifically characterized subset
%of $\OGamma_\Psi$ is introduced. Introducing the assumption that such a subset exists and replacing the discrete normal direction with an exact one, the condition \eqref{conds3A} is proven. The extension to the case with discrete normal directions is then provided in \ref{sec:approxnormal}. Finally, the previously introduced assumption is shown in section \ref{sec:as-local-tubular-neighborhoods}.

\subsection{The conditions \eqref{conds} for the normal derivative volume stabilization} \label{sec:proofs:easy}
\begin{lemma}
  If the scaling assumption \eqref{eq:as-rho} holds, the normal derivative volume stabilization satisfies condition \eqref{conds1A}.
\end{lemma}

\begin{proof}
Using the scaling assumption we get
\[
  s_h(w,w) = \rho_s\|n_h \cdot \nabla w\|_{L^2(\OGamma_\Theta)}^2 \lesssim h^{-1} \|\nabla w\|_{L^2(\OGamma_\Theta)}^2,
\]
and thus
\eqref{conds1A} holds.
\end{proof}

\begin{lemma}
  If the scaling assumption \eqref{eq:as-rho} holds, the normal derivative volume stabilization satisfies condition \eqref{conds2A} with $l=k$.
\end{lemma}

\begin{proof}
Using the Cauchy--Schwarz inequality and \eqref{eq:def-h-norm}, we obtain
\[
  \sup_{w_h\in \Vkz}\frac{s_h(u^e,w_h)}{\norm{w_h}_h}\le \rho_s^{\frac12}\norm{n_h\cdot\nabla u^e}_{L^2(\OGamma_\Theta)}.
\]
From $n\cdot\nabla u^e\equiv 0$, \eqref{normalres1} and \eqref{nb} we get
\[
  \norm{n_h\cdot\nabla u^e}_{L^2(\OGamma_\Theta)}\lesssim h^k\norm{\nabla u^e}_{L^2(\OGamma_\Theta)}\lesssim h^{k+\frac12}\norm{\nabla_\Gamma u}_{L^2(\Gamma)}.
\]
Together with the well-posedness of \eqref{eq:weak-LB}, this yields
\[
  \sup_{w_h\in \Vkz}\frac{s_h(u^e,w_h)}{\norm{w_h}_h}\lesssim \rho_s^{\frac12}h^{k+\frac12}\norm{\nabla_\Gamma u}_{L^2(\Gamma)}\lesssim \rho_s^{\frac12}h^{k+\frac12}\norm{f}_{L^2(\Gamma)}.
\]
The assertion follows from the upper bound for $\rho_s$ in \eqref{eq:as-rho}.
\end{proof}
\begin{remark} \rm From the proof above it follows that if $\rho_s \sim h$ the normal derivative volume stabilization satisfies condition \eqref{conds2A} with $l=k+1$.
\end{remark}
\begin{lemma} \label{lemma3A}
   If the scaling assumption \eqref{eq:as-rho} holds, the normal derivative volume stabilization satisfies \eqref{conds3A} for $h$ sufficiently small.
\end{lemma}

\begin{proof}
 The analysis is given in the next section, cf. Corollary \ref{coro:cond3A}. 
\end{proof}
%It remains to establish the two ingredients of the proof. As this requires some technicalities, we first discuss a motivational example. 
\subsection{Proof of Lemma~\ref{lemma3A}} \label{sectproof3A}
In the neighborhood $U_r$ of $\Gamma$ we use the following local coordinate system, cf.~\eqref{decompo}.  For $x \in U_r$ We write
\begin{equation} \label{loccoordinate} x=(\xi,s),~~~ \xi=p(x) \in \Gamma,~s = {\rm sdist}(x)\in [-r,r],~~\text{i.e., } x= \xi + s\, n(\xi).
\end{equation} 
Let $\gamma \subset \Gamma$ be a simply connected subdomain of $\Gamma$ with ${\rm meas_2}(\gamma) >0$ (e.g., $\gamma=\Gamma$). Below, we consider neighborhoods $U_\gamma$ of $\Gamma$ which have the form
\begin{equation} \label{defU}
 U_\gamma = \{ (\xi,s) \mid \xi \in \gamma, -g(\xi)h\le s\le G(\xi)h \},
\end{equation}
with scalar Lipschitz functions $g \ge 0$, $G\ge 0$. This means that $U_\gamma$ is bounded by the graphs of $g$ and $G$ over $\Gamma$ (when mapped in the normal direction), cf. the sketch in Figure \ref{fig:ass1} below.

The following lemma is of fundamental importance in our analysis.
\begin{lemma}\label{lemJoerg-local}
Let $U_\gamma$ be a set as in \eqref{defU} and assume $\|g+G\|_{L^\infty(\gamma)} \lesssim 1$. The following holds:
\begin{equation} \label{fundest}
 \|u\|_{L^2(U_\gamma)}^2 \lesssim h \|u\|_{L^2(\gamma)}^2 + h^2 \|n_\Gamma\cdot \nabla u\|_{L^2(U_\gamma)}^2 \quad \text{for all} ~~u \in H^1(U_\gamma).
\end{equation}
\end{lemma}
\begin{proof}
Let $u\in C^\infty(U_\gamma)$. For each $\xi\in\gamma$, let $F_\xi$ denote the line-segment $\{\xi +s n(\xi)\mid -g(\xi)h\le s\le G(\xi)h\}\subseteq U_\gamma$. From the fundamental theorem of integration, we get for each $x= \xi+s n (\xi)\in F_\xi$ that
\[
  u(x)^2= u(\xi)^2 + 2\int_0^s u(\xi+tn(\xi))n(\xi)\cdot\nabla u(\xi+tn(\xi)) \d t.
\]
The Cauchy--Schwarz inequality implies
\[
  u(x)^2\le u(\xi)^2 + 2\norm{u}_{L^2(F_\xi)}\norm{n\cdot\nabla u}_{L^2(F_\xi)}.
\]
We integrate over $x\in F_\xi$ and apply the inequality $2ab \le \frac12 r_\xi^{-1}a^2 + 2 r_\xi b^2$ with $r_\xi = \meas_1(F_\xi)$ to obtain
\[
  \norm{u}_{L^2(F_\xi)}^2\le 2r_\xi u(\xi)^2 + 4 r_\xi^2\norm{n\cdot\nabla u}_{L^2(F_\xi)}^2.
\]
To relate the integral over $U_\gamma$ and $F_\xi$, we use the coarea formula \cite{cuckerbook2013,federerbook1969} for the closest-point projector $p\colon U_\gamma\to \gamma \subset \Gamma$. We interpret $p$ as a map onto the 2-dimensional manifold $\gamma\subseteq\Gamma$. The ``level set'' of each $\xi\in\gamma$ is $F_\xi$. For a function $f\in C^\infty(U_\gamma)$, we get
\[
  \int_{U_\gamma} f \d x = \int_\gamma \int_{F_\xi} f J(x)^{-1} \d s \d \sigma(\xi),
\]
where $J(x)=J(\xi,s)$ is the so-called normal-Jacobian of $p(x)$. Elementary computations yield that $Dp(x)=P(\xi)(I +s H(\xi))^{-1} P(\xi)$ with $s={\rm sdist}(x)$ and the Hessian $H$ of ${\rm sdist}$, and $ J(\xi,s)^{-1}= \det(I +s H(\xi))$. From this we obtain $J(x)^{-1} \sim 1$.
%
%We show that $J(x)\sim 1$ holds on $\gamma$. An elementary computation yields $Dp(x)=P(\xi)(I +s H(\xi))^{-1} P(\xi)$ with $s={\rm sdist}(x)$ and the Hessian $H$ of %${\rm sdist}$. We introduce an orthonormal basis of $\R^3$ consisting of $n(\xi)$ and two vectors $t_1$, $t_2$ (depending on $\xi$) such that $t_i^T H(\xi) t_i = %\kappa_i$, $i=1,2$, is a principal curvature of $\gamma$ at the point $\xi$. In this basis, $Dp(x)$ is represented by the blockmatrix
%\[
%  \begin{pmatrix}
%    0 & 0 \\
%    0 & A
%  \end{pmatrix},
%  \quad
%  A=\operatorname{diag}((1+s\kappa_1)^{-1}, (1+s\kappa_2)^{-1}),
%\]
%and there holds $J(x)=\det A$. From \eqref{eq:lsetprops}, we obtain $\lvert \kappa_i\rvert\lesssim 1$, $i=1,2$. Using $\norm{g+G}_{L^\infty(\gamma)}\lesssim 1$ gives $\lvert s\rvert \le r_\xi \lesssim h$ and thus $J(x)= ((1+s\kappa_1)(1+s\kappa_2))^{-1}\sim 1$.

Inserting $f=u^2$ into the coarea formula, we obtain $\norm{u}_{L^2(U_\gamma)}^2 \lesssim \int_\gamma r_\xi u^2 + \int_\gamma r_\xi^2 \norm{n\cdot\nabla u}_{L^2(F_\xi)}^2$. Using $r_\xi \lesssim h$ yields the result in \eqref{fundest}. A density argument completes the proof.
\end{proof}

This lemma shows that one can control the $L^2$-norm in the volume $U_\gamma$ with the normal derivative in the same volume (as used in the stabilization) and the $L^2$-norm on the surface. The result in \eqref{fundest} can be interpreted as a ``local version'' of \eqref{conds3A}. Below we will use this in combination with a localization argument to obtain a result as in \eqref{conds3A} up to a geometric error ($\OGamma_\Psi$ vs. $\OGamma_\Theta$). This geometric error can be dealt with as shown in Lemma~\ref{lemJoergpert}.

\subsubsection{Localization argument}
\begin{figure}[h!]
  \vspace*{-0.2cm}
  \begin{center}
    \begin{tikzpicture}[scale=1.5,line cap=round]
      % the triangles
      \coordinate (A) at (0.0,0.0);
      \coordinate (B) at (1.0,0.0);
      \coordinate (C) at (2.0,0.0);
      \coordinate (D) at (3.0,0.0);
      \coordinate (E) at (0.0,1.0);
      \coordinate (F) at (1.0,1.0);
      \coordinate (G) at (2.5,2.0);
      \coordinate (H) at (2.0,1.0);
      \coordinate (I) at (3.0,1.0);
      \tikzstyle{mytriangles}=[line width=0.8pt, line join=round];
      \draw[mytriangles] (A) -- (F) -- (E) -- cycle;
      \draw[mytriangles] (A) -- (B) -- (F) -- cycle;
      \draw[mytriangles] (B) -- (G) -- (F) -- cycle;
      \draw[mytriangles] (B) -- (H) -- (G) -- cycle;
      \draw[mytriangles] (B) -- (C) -- (H) -- cycle;
      \draw[mytriangles] (C) -- (I) -- (H) -- cycle;
      \draw[mytriangles] (C) -- (D) -- (I) -- cycle;
      \node[right] at (D) {\color{black}$\OGamma_\Psi$};
      % the border of $\OGamma$
      \tikzstyle{myborder}=[line width=1pt, line cap=round, line join=round];
      \draw[myborder] (A) -- (B) -- (C) -- (D);
      \draw[myborder] (E) -- (F) -- (G) -- (H) -- (I);
      % the manifold
      \coordinate (MA) at (-0.5, 0.5);
      \coordinate (MB) at (3.5,0.5);
      \draw[line width=1pt,red!80!black] (MA) -- (MB) node[right]{$\Gamma$};
      % the unreachable patch
      \coordinate (P) at (2.0,1.6667);
      \coordinate (Q) at (2.5,0.5);
      \coordinate (R) at (2.5,1.0);
      \draw[line width=1pt, line join=round, green, fill=green,fill opacity=0.3] (H) -- (G) -- (P) -- cycle;
      \draw[line width=1pt, green]         (R) -- (Q);
      \draw[line width=1pt, green, dotted] (Q) -- (G);
    \end{tikzpicture}
  \end{center}
  \vspace*{-0.4cm}
  \caption{Depending on the shape of the $T\in\mT$, $\OGamma_\Psi$ is not a graph over $\Gamma$ as in \eqref{defU}.}
  \label{fig:obtuse-neighborhood}
\end{figure}
In general, we cannot expect $\OGamma_\Psi=U_\Gamma$ for some $U_\Gamma$ as in \eqref{defU}, cf.~Figure~\ref{fig:obtuse-neighborhood}. Therefore, we present a localization argument which is based on the following observation (lemma~\ref{la:norm-B} below): On the finite element space (as opposed to $H^1(U_\gamma)$ used in \eqref{fundest}), it suffices to control the $L^2$-norm on suitable subsets in order to get a bound on the $L^2(\OGamma_\Psi)$-norm. We apply the localization argument to the triangulation $\mT^\Gamma_\Psi =\{ \, \Psi(T)\mid T\in \mT^\Gamma \,\}$ because this triangulation corresponds to the globally smooth surface $\Gamma$, cf.~Fig.~\ref{fig:trafos}. On $\mT^\Gamma_\Psi$ we define, with $F_h:=\Theta_h \circ \Psi^{-1}$ (cf. Fig.~\ref{fig:trafos}), the finite element space
 \[
  V_{h,\Psi}^k:= \{\, \tilde u_h = u_h \circ F_h \mid u_h \in \Vk \,\}.
 \]
Let $\{B_T\mid T\in \mT^\Gamma_\Psi\}$ be a collection of balls with $B_T\subset T$ and ${\rm radius}(B_T)=: r_T\gtrsim h$ for all $T\in \mT^\Gamma_\Psi$.
%An example would be the collection of all inballs of the $\mT^\Gamma_\Psi$; the inballs satisfy $r_T\gtrsim h$ because of quasi-uniformity.
Let
\begin{equation}\label{eq:norm-B}
  \|u\|_B^2= \sum_{T\in\mT^\Gamma_\Psi} \norm{u}_{L^2(B_T)}^2.
\end{equation}

\begin{lemma}\label{la:norm-B}
  On $V_{h,\Psi}^k$ the uniform norm equivalence $\|\cdot\|_B \sim \|\cdot\|_{L^2(\OGamma_\Psi)}$ holds.
\end{lemma}\begin{proof}
As $B_T\subset T$ holds for all $T\in\mT^\Gamma_\Psi$, we immediately find $\|u\|_B \le \|u\|_{L^2(\OGamma_\Psi)}$ for all $u\in V_{h,\Psi}^k$.

To prove the other direction of the estimate, let $u\in V_{h,\Psi}^k$ and $T\in\mT^\Gamma_\Psi$ be arbitrary. We can write $T=\Psi(S)$, $S\in\mTGamma$. Furthermore, $u|_T=\hat u\circ \Psi^{-1}|_T$ for some polynomial $\hat u$ of degree $k$. Using $\Psi= I + \mathcal{O}(h)$, cf. \eqref{eq:phi2}, it follows that there is a ball $S_T \subset \Psi^{-1}(B_T)$ with ${\rm radius}(S_T) \gtrsim h$. From standard finite element analysis we obtain $\|\hat u\|_{L^2(S)} \sim \|\hat u\|_{L^2(S_T)}$. Using this and 
$
  \norm{u}_{L^2(T)}\sim \norm{\hat u}_{L^2(S)} $, which follows from \eqref{eq:phi2}, we get
\[
 \| u\|_{L^2(B_T)} \sim \|\hat u\|_{L^2(S_T)}\sim \norm{u}_{L^2(T)},
\]
and summing over $T\in\mT^\Gamma_\Psi$ completes the proof.
\end{proof}

The following assumption specifies quantitatively that each $T\in\mT^\Gamma_\Psi$ contains a sufficiently big ball which is ``locally visible'' from $\Gamma$ in a set as in \eqref{defU}.

\begin{assumption}\label{as:local-tubular-neighborhoods}
  For each $T\in\mT^\Gamma_\Psi$ there exists a set $U_{\gamma_T}$ as in \eqref{defU} with the following properties. The graph functions $g_T\ge 0$, $G_T\ge 0$ on $\gamma_T$ satisfy $\norm{g_T+G_T}_{L^\infty(\gamma_T)}\lesssim 1$. Furthermore, $U_{\gamma_T} \subseteq \OGamma_\Psi$, ${\rm diam}(\gamma_T) \lesssim h$ and $U_{\gamma_T}$ contains a ball $B_T\subset T$ with radius $r_T\gtrsim h$.
\end{assumption}

For a sketch of the domains involved in Assumption \ref{as:local-tubular-neighborhoods} we refer to Figure \ref{fig:ass1}.

\begin{figure}
  \begin{center}
    \begin{tikzpicture}[scale=0.45]
       % vertices of the mesh
       \coordinate (A) at (0.5,6.1);
       \coordinate (B) at (3.5,0.8);
       \coordinate (C) at (6.2,5.5);
       \coordinate (D) at (9.7,0.6);
       \coordinate (E) at (12.3,5.5);
       \coordinate (F) at (1.0,0.7);
       % the curved mesh:
       \draw[line width=1pt] (D)
       to[in=6,out=182] (B)
       to[in= -90,out= 135] (A)
       to[in= 166,out=  -2] (C)
       to[in= 180,out=  14] (E)
       to[in=  70,out=-115] (D)
       to[in= -75,out= 135] (C)
       to[in=  83,out=-100] (B);
       % marked current element T
       \draw[fill=gray,opacity=0.2, draw=none] (D)
       to[in= -75,out= 135] (C)
       to[in=  83,out=-100] (B)
       to[in= 186,out=   2] (D) -- cycle;

       % some points on the xi-axis
       \coordinate (XA) at (-.25, 3.5);
       \coordinate (XB) at ( 2.5, 3.5);
       \coordinate (XC) at (10.0, 3.5);
       \coordinate (XD) at (12.0, 3.5);
       % drawing the axis
       \draw[line width=1pt,red!80!black,->] (XA) -- (XD) node[right]{$\xi$};
       % drawing \gamma_T
       \draw[line width=2pt,blue!80!black] (XB) -- (XC);

       % lower control points for U_{\gamma_T}
       \coordinate (YA) at ( 5.6, 1.4);
       \coordinate (YB) at ( 7.6, 2.1);
       % upper control points for U_{\gamma_T}
       \coordinate (ZA) at ( 5.0, 4.5);
       \coordinate (ZB) at ( 6.8, 4.9);
       % drawing U_{\gamma_T}
       \draw[fill=green,opacity=0.3] (XB)
       to[ in= 160, out= -35] (YA)
       to[out= -20,  in=-130] (YB)
       to[out=  50,  in=-174] (XC)
       to[ in=  45, out= 180] (ZB)
       to[ in=   0, out=-135] (ZA)
       to[ in=  22, out= 180] (XB);
       
       % length display of \gamma_T
       \coordinate (BXB) at ( 2.5, 0.0);
       \coordinate (BXC) at (10.0, 0.0);
       \coordinate (BXM) at (6.25, 0.0);
       % drawing the lengts with node
       \draw[line width=0.75pt,blue!80!black,densely dotted,<->] (BXB) -- (BXC);
       \node[below,scale=0.8] at (BXM)
       {\color{blue!80!black} $\operatorname{diam}(\gamma_T)\lesssim h$};

       % B_T and control points for drawing the diam
       \coordinate (XX)  at ( 6.08, 2.6);
       \coordinate (XXA) at ( 4.83, 2.6);
       \coordinate (XXB) at ( 7.33, 2.6);
       % drawing the circle
       \filldraw[orange,opacity=0.3,draw=black] (XX) circle (1.25);
       \draw[line width=0.75pt,orange!50!black,densely dotted,<->] (XXA) -- (XXB);
       \node[below,scale=0.8] at (XX)
       {\color{orange!50!black} $\gtrsim h$};
       
       % control points for bounding g_T and G_T
       \coordinate (GAA) at (-0.25, 5.15);
       \coordinate (GAB) at ( 12.0, 5.15);
       \coordinate (GBA) at (-0.25, 1.3);
       \coordinate (GBB) at ( 12.0, 1.3);
       % drawing bounds for g_T and G_T
       \draw[line width=0.75pt,black,densely dotted] (GAA) -- (GAB);
       \draw[line width=0.75pt,black,densely dotted] (GBA) -- (GBB);
       \draw[line width=0.75pt,black,densely dotted,<->] (GAA) -- (GBA);
       \node[left,scale=0.8]  at (XA)
       { $\Vert g_T + G_T\Vert_{L^\infty({\color{blue!80!black}}\gamma_T)} \cdot h \lesssim h$};

       % labels / legend:
       %  label Gamma
       \coordinate (LR1) at ( 14.5, 5.5);
       \coordinate (LR2) at ( 15.5, 5.5);
       \draw[line width=1pt,red!80!black] (LR1) -- (LR2) node[right]{$\Gamma$};
       %  label gamma_T
       \coordinate (LB1) at ( 14.5, 4.5);
       \coordinate (LB2) at ( 15.5, 4.5);
       \draw[line width=2pt,blue!80!black] (LB1) -- (LB2) node[right]{$\gamma_T$};
       %  label U_{gamma_T}
       \coordinate (LU) at ( 15, 3.4);
       \coordinate (LUR) at ( 15.5, 3.3);
       \filldraw[green,opacity=0.3,draw=black] (LU) circle (0.5);
       \node[right] at (LUR) {\color{green!50!black} $U_{\gamma_T}$};
       %  label B_T
       \coordinate (LBT) at ( 15, 2.1);
       \coordinate (LBTR) at ( 15.5, 2.0);
       \filldraw[orange,opacity=0.3,draw=black] (LBT) circle (0.5);
       \node[right] at (LBTR) {\color{orange!50!black} $B_T$};
       %  label T
       \coordinate (LT) at ( 15, 0.8);
       \coordinate (LTR) at ( 15.5, 0.7);
       \filldraw[gray,opacity=0.2,draw=black] (LT) circle (0.5);
       \filldraw[fill=none,draw=black] (LT) circle (0.5);
       \node[right] at (LTR) {\color{gray!50!black} $T$};
     \end{tikzpicture} \vspace*{-0.25cm}
   \end{center}
   \caption{Sketch of the domains involved in Assumption \ref{as:local-tubular-neighborhoods}.}
   \label{fig:ass1}
\end{figure}
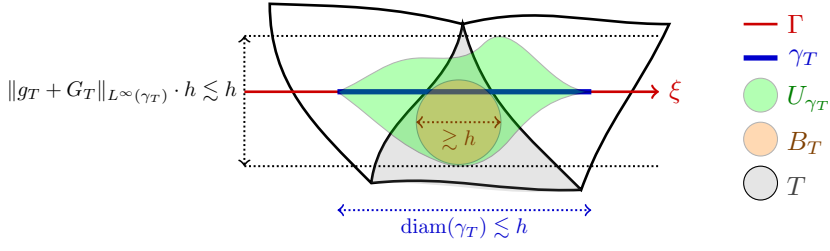

\begin{lemma}\label{lemJoerg}
  If Assumption \ref{as:local-tubular-neighborhoods} is satisfied, the following holds:
  \[
    \norm{u}_{L^2(\OGamma_\Psi)}^2 \lesssim h\norm{u}_{L^2(\Gamma)}^2 + h^2\norm{n_\Gamma\cdot\nabla u}_{L^2(\OGamma_\Psi)}^2 \quad\text{for all } u\in V_{h,\Psi}^k.
  \]
\end{lemma}

\begin{proof}
Let $u\in V_{h,\Psi}^k$ be arbitrary. From Lemma \ref{la:norm-B} and Assumption \ref{as:local-tubular-neighborhoods}, we get
\[
  \norm{u}_{L^2(\OGamma_\Psi)}^2 \lesssim \norm{u}_B^2\le \sum_{T\in\mTGamma_\Psi} \norm{u}_{U_{\gamma_T}}^2.
\]
We apply Lemma \ref{lemJoerg-local} on each $T\in\mTGamma_\Psi$,
\[
  \sum_{T\in\mTGamma_\Psi} \norm{u}_{U_{\gamma_T}}^2 \lesssim \sum_{T\in\mTGamma_\Psi}\left( h\norm{u}_{L^2(\gamma_T)}^2 + h^2\norm{n_\Gamma\cdot\nabla u}_{L^2(U_{\gamma_T})}^2 \right).
\]
Due to ${\rm diam}(\gamma_T) \lesssim h$, cf. Assumption \ref{as:local-tubular-neighborhoods}, we can apply a standard finite intersection argument. Hence the right-hand side of the previous estimate is uniformly bounded by \vspace*{-0.25cm}
\[
  \sum_{T\in\mTGamma_\Psi}\left( h\norm{u}_{L^2(\Gamma\cap T)}^2 + h^2\norm{n_\Gamma\cdot\nabla u}_{L^2(T)}^2 \right).\vspace*{-0.25cm}
 % \vspace*{-1cm} \vspace*{-0.25cm}
\]
\end{proof}

\noindent
Finally we treat Assumption \ref{as:local-tubular-neighborhoods}:
%\clnote{I checked the check and modified the picture where I changed the right part of Figure 7.3 only. I am not sure which corrections are desired apart from those.}
%\jgnote{I like the Figures as they are now.}
\begin{lemma}\label{la:as-local-tubular-neighborhoods}
  On a quasi-uniform family of triangulations, for sufficiently small $h$, Assumption \ref{as:local-tubular-neighborhoods} holds.
\end{lemma}
\begin{proof}
%By $h_{\min}$ we denote the minimal radius of the (maximal) inscribed spheres over all $T \in \mT^\Gamma_\Psi$. Due to quasi-uniformity we have $h \sim h_{\min}$. %The local intersection is denoted by $\Gamma_T:= \Gamma \cap T$, $T \in \mT^\Gamma_\Psi$.
%We assume that $h$ is sufficiently small such that $\Gamma_T$ is ``flat'' in the following sense. Due to smoothness of $\Gamma$ we can take $h$ small such that there is a 2D plane $P$, which intersects $\Gamma$, and $\Gamma_T$ can be represented as the graph of a function $g_T: P_T \to \mathbb{R}^3$, $P_T \subset P$. We assume $h$ small enough such that for some fixed $\epsilon$ with $0 < \epsilon \ll 1$ we have $\|\nabla_p g_T\|_{L^\infty(P_T)} \leq \epsilon h_{\min}$. 
The idea of the proof is as follows. For each $T \in \mT^\Gamma_\Psi$ we will 
 define a set $U_{\gamma_T}$ as in \eqref{defU}, which equals a ``half-ball'' in the local coordinates \eqref{loccoordinate} with radius $\sim h$. In the construction  we will distinguish two cases, namely either $\Gamma_T:=\Gamma \cap T$ is ``close to'' $\partial\mT^\Gamma_\Psi$ or this is not the case, cf. Fig.~\ref{fig:sketchproofass1}. 
%as:local-tubular-neighborhoods} are satisfied by taking $U_{\gamma_T}$ equal to this sphere, whereas in the latter case we take $U_{\gamma_T}$ equal to a suitably defined intersection of the sphere with a half-space, which then satisfies all requirements.

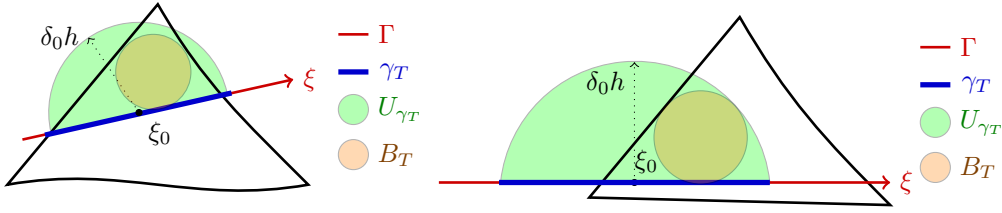
\begin{figure}
     \begin{tikzpicture}[scale=2]
       \coordinate (A) at (0.0,0.0);
       \coordinate (B) at (2.0,0.0);
       \coordinate (C) at (1.0,1.2);

       \coordinate (D) at (0.1,0.3);
       \coordinate (E) at (1.9,0.7);
       
       \coordinate (F) at (0.875,0.48);

       \coordinate (G) at (0.25,0.33);
       \coordinate (H) at (1.49,0.61);
       \coordinate (I) at (0.97,0.75);

       \coordinate (LR1) at (2.2,1.0);
       \coordinate (LR2) at (2.4,1.0);

       \coordinate (LB1) at (2.2,0.75);
       \coordinate (LB2) at (2.4,0.75);

       \coordinate (LU)  at (2.3,0.5);
       \coordinate (LBT) at (2.3,0.2);
       \coordinate (LUR)  at (2.4,0.5);
       \coordinate (LBTR) at (2.4,0.2);
       
       \draw[fill=green,opacity=0.3] (F) -- ++(13:.6) arc (13:193:.6);

       \draw[dotted,->] (F) -- ++(124:.6) node[left]{$\delta_0 h$};
       
       \draw[line width=1pt] (A) to[in=-170,out=10] (B) to[in=-60,out=135] (C) -- cycle;
       \draw[line width=1pt,red!80!black,->] (D) -- (E) node[right] {$\xi$};
       \draw[line width=2pt,blue!80!black] (G) -- (H);
       \filldraw[orange,opacity=0.3,draw=black] (I) circle (0.25);

       \draw[line width=1pt,red!80!black] (LR1) -- (LR2) node[right]{$\Gamma$};
       \draw[line width=2pt,blue!80!black] (LB1) -- (LB2) node[right]{$\gamma_T$};
       \filldraw[green,opacity=0.3,draw=black] (LU) circle (0.1);
       \node[right] at (LUR) {\color{green!50!black} $U_{\gamma_T}$};
       \filldraw[orange,opacity=0.3,draw=black] (LBT) circle (0.1);
       \node[right] at (LBTR) {\color{orange!50!black} $B_T$};
       \filldraw[draw=black] (F) circle (0.02) node[below right]{$\xi_0$};
     \end{tikzpicture}
     \begin{tikzpicture}[scale=2]
       \coordinate (A) at (0.0,0.0);
       \coordinate (AL) at (0.0,-0.1);
       \coordinate (AU) at (0.3,0.1);
       \coordinate (B) at (2.0,-0.05);
       \coordinate (C) at (1.0,1.2);
       
       \coordinate (D) at (-1.0,0.1);
       \coordinate (E) at (2.0,0.1);
       
       \coordinate (F) at (1.0,0.7);

       \coordinate (G) at (-0.6,0.1);
       \coordinate (H) at (1.2,0.1);
       
       \coordinate (I) at (0.74,0.4);

       \coordinate (LR1) at (2.2,1.0);
       \coordinate (LR2) at (2.4,1.0);

       \coordinate (LB1) at (2.2,0.75);
       \coordinate (LB2) at (2.4,0.75);

       \coordinate (LU)  at (2.3,0.5);
       \coordinate (LUR)  at (2.4,0.5);
       \coordinate (LBT) at (2.3,0.2);
       \coordinate (LBTR) at (2.4,0.2);

       \draw[fill=green,opacity=0.3] (AU) -- ++ (0.9,0.0) arc (6:173.5:0.9) -- cycle;

       \draw[dotted,->] (AU) -- (0.3,0.9) node[below left]{$\delta_0 h$};
       \filldraw[draw=black] (AU) circle (0.02) node[above right]{$\!\!\xi_0$};
       
       \draw[line width=1pt] (A) -- (B) to[in=-60,out=135] (C) -- cycle;
       \draw[line width=1pt,red!80!black,->] (D) -- (E) node[right]{$\xi$};
       \draw[line width=2pt,blue!80!black] (G) -- (H);
       \filldraw[orange,opacity=0.3,draw=black] (I) circle (0.31);

       \draw[line width=1pt,red!80!black] (LR1) -- (LR2) node[right]{$\Gamma$};
       \draw[line width=2pt,blue!80!black] (LB1) -- (LB2) node[right]{$\gamma_T$};
       \filldraw[green,opacity=0.3,draw=black] (LU) circle (0.1);
       \node[right] at (LUR) {\color{green!50!black} $U_{\gamma_T}$};
       \filldraw[orange,opacity=0.3,draw=black] (LBT) circle (0.1);
       \node[right] at (LBTR) {\color{orange!50!black} $B_T$};
     \end{tikzpicture}
     \caption{Sketch of the two cases: There either exists $\xi_0 \in \Gamma$, so that
       $\operatorname{dist}_{\ast}(\xi_0,\partial \Omega_{\Psi}^\Gamma) > \delta_1 h_{\min}$ (left) or $\operatorname{dist}_{\ast}(\Gamma,\partial \Omega_{\Psi}^\Gamma) < \delta_1 h_{\min}$ (right). According to the two cases the sets $U_{\gamma_T}$, $\gamma_T$ and $B_T$ can be found. Note that the sketch uses the local coordinates $(\xi,s)$ as in \eqref{defU}.}
     \label{fig:sketchproofass1}
\end{figure}

First we introduce (small) balls in the local coordinates. We define the distance ${\rm dist}_\ast\big((\xi,s),(\tilde \xi,\tilde s)\big):= \big( \|\xi - \tilde \xi\|_2^2 + |s-\tilde s|^2\big)^\frac12$, where $\|\cdot\|_2$ is the Euclidean distance and $(\xi,s)$ are the local coordinates as in \eqref{loccoordinate}.
This distance is equivalent to the Euclidean distance: there are constants $d_0 >0$ and $d_1$ such that for all $x=(\xi,s)$ and $\tilde x=(\tilde \xi, \tilde s)$ from $U_r$ we have 
\begin{equation} \label{normeq}
 d_0 \|x- \tilde x\|_2 \leq {\rm dist}_\ast\big((\xi,s),(\tilde \xi,\tilde s)\big) \leq d_1 \|x- \tilde x\|_2 .
\end{equation}
% \clnote{Proof this.}
In this distance the balls with center $\xi_0 \in \Gamma$ are denoted by $B_\ast(\xi_0;\delta):= \{\, (\xi,s)\mid {\rm dist}_\ast \big((\xi_0,0),(\xi,s) \big) \leq \delta\,\}$. For defining  suitable half-balls we introduce some further notation. We define the part of the domain $\Omega^\Gamma_\Psi$ with negative (positive) level set values and the corresponding part of the outer boundary:
\begin{equation*}
  \OGamma_{\Psi,\mp} := \{\, x \in \OGamma_\Psi \mid \phi(x) \lessgtr 0\,\}, \quad 
 \Gamma_{\pm} := \partial\OGamma_{\Psi,\pm} \setminus \Gamma.
\end{equation*}
 For $\xi_0 \in \Gamma$ we define the ``half-balls''
 $B_\ast^{\pm}(\xi_0;\delta):=B_\ast(\xi_0;\delta)\cap \OGamma_{\Psi,\pm}$.
 Using the quasi-uniformity assumption on the family of triangulations one can show 
 that $\min\{\| x_+ - x_-\|_2 \mid x_\pm\in \Gamma_\pm\} \gtrsim h$. Hence, also
 $\min\{{\rm dist}_\ast((\xi_+, s_+), (\xi_-, s_-)) \mid (\xi_\pm, s_\pm)\in \Gamma_\pm\} \gtrsim h$ holds.
 Using this we conclude that there exists a $\delta_0 >0$ (independent of $h$) such that for all $\xi_0 \in \Gamma$
 \begin{equation} \label{either}
  B_\ast^+(\xi_0;\delta_0 h) \subset \OGamma_{\Psi,+} \quad \text{or}~~ B_\ast^-(\xi_0;\delta_0 h) \subset \OGamma_{\Psi,-}
 \end{equation}
 holds. In the remainder we take such a fixed $\delta_0 > 0$. One checks that such half-balls $B_\ast^\pm(\xi_0;\delta_0 h)$  are of the form $U_\gamma$ as in \eqref{defU}, with $\gamma:=\{\, \xi \in \Gamma \mid \|\xi- \xi_0\|_2 \leq \delta_0 h \,\}$ and $\|g+G\|_{L^\infty(\gamma)} \leq  \delta_0$, ${\rm diam (\gamma)} \lesssim h$. 

 Take $T \in \mT^\Gamma_\Psi$. We now show that such half-balls with center $\xi_0 \in \Gamma_T=\Gamma \cap T$
 satisfy the conditions required in Assumption \ref{as:local-tubular-neighborhoods}. Note that there exists a constant $\hat c>0$, depending only on the shape regularity of the triangulation, such that for all $\xi \in \Gamma_T$ we have $|B_\ast(\xi;\delta_0 h) \cap T| \geq \hat c h^3$. Hence
\begin{equation}\label{either1}
 |B_\ast^+(\xi;\delta_0 h) \cap T| \geq \frac12 \hat c h^3 \quad \text{or}~~|B_\ast^-(\xi;\delta_0 h) \cap T| \geq \frac12 \hat c h^3
\end{equation}
holds.
We introduce the following boundary strip. For fixed $\delta_1$ with $0< \delta_1 \leq \frac12 \delta_0$ we define
 \[
  \Gamma_{\delta_1}:= \{\, x \in \OGamma_\Psi \mid {\rm dist}_\ast(x, \partial\OGamma_\Psi) \leq \delta_1 h\,\}.
 \]
Then either $\Gamma_T \subset \Gamma_{\delta_1}$ or there exists $\xi_0 \in \Gamma_T$ with ${\rm dist}_\ast((\xi_0,0), \partial\OGamma_\Psi) > \delta_1 h$.
We first consider the latter case. By construction we have that both half-balls $B_\ast^+(\xi_0;\delta_1 h)$ and  $B_\ast^-(\xi_0;\delta_1 h)$ are  contained in  $\OGamma_\Psi$. We choose one of these, say $U_{\gamma_T}:=B_\ast^+(\xi_0;\delta_1 h) \subset \OGamma_\Psi$, such that $|U_{\gamma_T}\cap T| \gtrsim h^3$ holds, see Figure \ref{fig:sketchproofass1} (left) for a sketch.
\\
We now consider the case $\Gamma_T \subset \Gamma_{\delta_1}$. Take a $\xi_0 \in \Gamma_T$, hence ${\rm dist}_\ast((\xi_0,0), \partial\OGamma_\Psi) \leq \delta_1 h$.
Without loss of generality we assume that $\xi_0$ is closest to $\Gamma_-= \partial\OGamma_{\Psi,-}$, and thus, cf.~\eqref{either}, $B_\ast^+(\xi_0;\delta_0 h) \subset \OGamma_{\Psi,+}$. Note that
\[
 |B_\ast^-(\xi_0;\delta_0 h)\cap T| \leq |\Gamma_{\delta_1} \cap T| \leq \delta_1 h {\rm diam}(T)^2 \leq \delta_1 h^3. 
\]
Using this and taking $\delta_1:=\frac14 \min\{\delta_0, \hat c\}$, we conclude from \eqref{either1} that for $U_{\gamma_T}:=B_\ast^+(\xi_0;\delta_0 h) \subset \OGamma_\Psi $ we have
$|U_{\gamma_T} \cap T| \gtrsim h^3$ holds.

In both cases we have have $U_{\gamma_T}=B_\ast^+(\xi_0;\delta_i h) \subset \OGamma_\Psi$ ($i=0$ or $i=1$), with $ |U_{\gamma_T} \cap T| \gtrsim h^3$.
Due to shape regularity we can construct a ball with radius $r_T \gtrsim h$ and $B_T \subset T \cap U_{\gamma_T}$. Hence, for this $U_{\gamma_T}$ all conditions in Assumption \ref{as:local-tubular-neighborhoods} are satisfied. 

%We now consider the case $\Gamma_T \subset \Gamma_{\delta_1}$ and take the vertex $x_V = (\xi_V,s_V)$ of $T$ which is closest to $\Gamma_T$. 
%We assume $(\xi_V,s_V) \in \Omega_{\Psi,-}^\Gamma$ and thus $B_\ast^+(\xi_V;\delta_0 h) \subset \OGamma_{\Psi,+}$, cf. ~\eqref{either} (the other case can be treatby the same arguments).We define $U_{\gamma_T} := B_\ast(x_V,\delta_0 h) \cap \Omega_{\Psi,+}^\Gamma \subset B_\ast^+(\xi_V;\delta_0 h)$
%and have
%$ |U_{\gamma_T} \cap T| \geq |B_\ast(\xi_V,\delta_0 h) \cap T| - | \Gamma_{\delta_1} \cap T|$, see Figure \ref{fig:sketchproofass1} (right) for a sketch. Due to shape regularity of the mesh we have that $|B_\ast(\xi_V,\delta_0 h) \cap T| \gtrsim h^3$ and $| \Gamma_{\delta_1} \cap T| \lesssim \delta_1 h^3$, so that for sufficiently small $\delta_1$ we can guarantee $|U_{\gamma_T} \cap T| \gtrsim h^3$ and thus we can construct a ball with radius $r_T \gtrsim h$ and $B_T \subset T \cap U_{\gamma_T}$. Hence, for this $U_{\gamma_T}$ all conditions in Assumption \ref{as:local-tubular-neighborhoods} are satisfied. 
\end{proof}

%For obtaining a satisfactory lower bound the result in the next lemma is a key ingredient. For this result we need some notation. We use the local coordinate system $(s,d)$ in a neighborhood $U_r$ of $\Gamma$, with $s \in \Gamma$, $d\in \mathbb{R}$ such that $(s,d)=x=s+dn(s)$ (..more explanaton..?). We assume that $\tilde U_h$ is a tubular neighborhood of $\Gamma$ of uniform width $\sim h$, but we do not require ${\rm dist}(\partial \tilde U_h,\Gamma)\sim h$ on \emph{both} sides of $\Gamma$, i.e.:
%\begin{equation} \label{defU}
% \tilde U_h = \{ x=(s,d(s)) \in U_r~|~ -g_h(s)h \leq d(s) \leq G_h(s) h \quad \text{for all}~~s\in \Gamma\, \},
%\end{equation}
%with scalar Lipschitz functions $g_h \geq 0,G_h\geq 0$ such that $\min_{s\in\Gamma}(g_h(s)+G_h(s)) \geq c_0 >0$ with a constant $c_0$ independent of $h$.
%\begin{lemma} \label{lemJoerg}
% Let $\tilde U_h$ be as in \eqref{defU}. The following holds:
%\[
% \|u\|_{L^2(\tilde U_h)}^2 \lesssim h \|u\|_{L^2(\Gamma)}^2 + h^2 \|n \cdot \nabla u\|_{L^2(\tilde U_h)}^2 \quad \text{for all}~~u \in H^1(\tilde U_h).
%\]
%\end{lemma}
%\begin{proof} 
%\end{proof}
%..discuss this result...\\

\subsubsection{Geometric error} \label{sec:approxnormal}
In this section we treat the geometric error ($\OGamma_\Psi$ vs. $\OGamma_\Theta$) by a straightforward perturbation argument.
We assume that we have a quasi-uniform family of triangulations, hence Assumption \ref{as:local-tubular-neighborhoods} is satisfied (for $h$ sufficiently small).
\begin{lemma} \label{lemJoergpert}
 Let $\tilde n$ be such that $\|\tilde n - n_\Gamma\|_{\infty,\OGamma_\Theta} \lesssim h$ holds. For $h$ sufficiently small, the following holds:
  \begin{equation} \label{estpert}
    \|u_h\|_{L^2(\OGamma_\Theta)}^2 \lesssim h \|u_h\|_{L^2(\Gamma_h)}^2 + h^2 \|\tilde n \cdot \nabla u_h\|_{L^2(\OGamma_\Theta)}^2 \quad \text{for all}~~u_h \in V_{h,\Theta}^k.
  \end{equation}
\end{lemma}

\begin{proof}
We use the homeomorphism $F_h = \Theta_h \circ \Psi^{-1}: \OGamma_\Psi \to \OGamma_\Theta$ (see also Figure \ref{fig:trafos}) which satisfies
\begin{equation}\label{lemJoergpert-1}
  \|I- F_h \|_{\infty,\OGamma_\Psi} + h\|I- DF_h \|_{\infty,\OGamma_\Psi} \lesssim h^{k+1},
\end{equation}
cf.~Lemma~\ref{lem:phih}. Take $u_h \in V_{h,\Theta}^k$ and define $\tilde u_h:=u_h \circ F_h\in V_{h,\Psi}^k$. Using standard transformation rules and the result in Lemma~\ref{lemJoerg} we obtain
\begin{align*}
  \|u_h\|_{L^2(\OGamma_\Theta)}^2 & \sim \|\tilde u_h \|_{L^2(\OGamma_\Psi)}^2 \lesssim h \|\tilde u_h\|_{L^2(\Gamma)}^2 + h^2 \|n_\Gamma \cdot \nabla \tilde u_h\|_{L^2(\OGamma_\Psi)}^2 \\
  & \sim h \| u_h\|_{L^2(\Gamma_h)}^2 + h^2 \|(DF_h n_\Gamma)\circ F_h^{-1} \cdot \nabla u_h\|_{L^2(\OGamma_\Theta)}^2 .
\end{align*}
From a triangle-inequality, $\|\tilde n - n_\Gamma\|_{\infty,\OGamma_\Theta} \lesssim h$ and \eqref{lemJoergpert-1} we get, for $h$ sufficiently small:
%\[
% \|n\circ F_h^{-1} - \tilde n\|_{\infty, \OGamma_\Theta} \leq \|n\circ F_h^{-1} - n\|_{\infty, \OGamma_\Theta} + \|n- \tilde n\|_{\infty, \OGamma_\Theta} \lesssim
% h.
%\]
\[
 \|(DF_h n_\Gamma)\circ F_h^{-1} - \tilde n\|_{\infty, \OGamma_\Theta} \lesssim
 h.
\]
Hence we obtain, using an inverse inequality:
\[
 \|u_h\|_{L^2(\OGamma_\Theta)}^2 \lesssim h \| u_h\|_{L^2(\Gamma_h)}^2 + h^2 \| \tilde n \cdot \nabla u_h\|_{L^2(\OGamma_\Theta)}^2 + h^2 \|u_h\|_{L^2(\OGamma_\Theta)}^2.
\]
For $h$ sufficiently small, we can adsorb the last term on the right-hand side in the term on the left hand-side, and this completes the proof.
\end{proof}

On $\Gamma$, there holds the Poincar{\'e} inequality $\norm{u}_{L^2(\Gamma)}\lesssim \norm{\nabla_\Gamma u}_{L^2(\Gamma)}$ for all $u\in H^1_\ast(\Gamma)$. Using the properties of the mapping $F_h\colon \OGamma_\Theta\to\OGamma_\Psi$ in Lemma \ref{lem:phih}, one can derive a Poincar{\'e} inequality on $\Gamma_h$ (see e.g. \cite[Remark 5.3]{reusken2015}), 
\begin{equation}\label{eq:poincare-gamma-h}
  \norm{u_h}_{L^2(\Gamma_h)}\lesssim \norm{\nabla_{\Gamma_h} u_h}_{L^2(\Gamma_h)}\quad \text{for all } u_h \in \Vkz.
\end{equation}

\begin{corollary}\label{coro:cond3A} If the scaling assumption \eqref{eq:as-rho} holds, the normal derivative volume stabilization satisfies \eqref{conds3A} for $h$ sufficiently small. 
\end{corollary}
\begin{proof} Take $\tilde n=n_h$ as defined in Lemma \ref{Lemnormals}, cf.~\eqref{normalres1}. Hence, $\|\tilde n - n_\Gamma\|_{\infty,\OGamma_\Theta} \lesssim h$ holds. Using this, Lemma \ref{lemJoergpert} and $h\lesssim \rho_s$ we get
\[
  h^{-1} \| u_h\|_{L^2(\OGamma_\Theta)}^2 \lesssim \|u_h\|_{L^2(\Gamma_h)}^2 + s_h(u_h, u_h) \quad \text{for all}~~u_h\in \Vk.
\]
The assertion follows from the Poincar{\'e} inequality \eqref{eq:poincare-gamma-h} and $\norm{\nabla_{\Gamma_h} u_h}_{L^2(\Gamma_h)}^2 = a_h(u_h,u_h)$.
\end{proof}

%{\bf Note AR: we can (only) guarantee that the stiffness matrix is invertible ``for $h$ sufficiently small''; can we show (under reasonable conditions) ${\rm cond}_\ast(\bS) < \infty$ for all $h$?}
%
%

\section{Numerical example} \label{sec:numex}
In this section we present numerical results for the isoparametric trace FEM explained in section~\ref{sec:itracefem} with a stabilization $s_h(\cdot,\cdot)$ as in section \ref{sec:normal-derivative-volume}. 
We first briefly discuss how we solve the linear systems arising from the discretization of the Laplace--Beltrami operator on the finite element spaces $\Vk$. The linear systems are singular because $\Vk$ contains constant functions.

\begin{remark}[Solution of (singular) linear systems] \label{sec:solvelinsys}\rm
Let $\bS \in \rr^{n \times n}$ be the stiffness matrix arising from the discretization such that 
$\bS_{i,j} = A_h(\varphi_j,\varphi_i)$ for basis functions $\varphi_i,\varphi_j$ of $\Vk$, $i,j \in \{1,..,n\}$, $n = \mathrm{dim}(V_{h,\Theta}^k)$, cf. section \ref{sec:condition-number}.
We seek a solution of 
\[
  \bS \bu = \bff \text{ with } \bu \text{ subject to } \langle \bc, \bu \rangle = 0,
\]
cf. \eqref{cdef}.
Here $\bu \in \rr^n$ denotes the coefficient vector of the solution such that the discrete solution is $u_h = \sum_{i=1}^n \bu_i \varphi_i$, $\bff \in \rr^n$ denotes the right-hand side functional, and $\bc \in \rr^n$ describes the constraint that the solution should be mean value free, cf.~\eqref{cdef}.
%(with respect to $\Gamma_h$), $\bc_i = \int_{\Gamma_h} \varphi_i \d{s_h}$, $i=1,..,n$.
We denote the coefficient vector of the discrete function which is (constant) one on $\Theta_h(\Omega_h^\Gamma)$ by $\be$ and note that $\mathrm{ker}(\bS) = \mathrm{span}(\be)$. 
Note that $\bc$ represents a functional (in $(V_{h,\Theta}^k)'$) whereas $\be$ represents a discrete function (in $V_{h,\Theta}^k$). There holds $\bc = \bM_\Gamma \be$ with the $L^2(\Gamma_h)$-mass-matrix $\bM_\Gamma$ of $V_{h,\Theta}^k$.

% {AR: something is confusing here? $\bc \neq \be $?)
%With the surface mass matrix $\bM^\Gamma\in\rr^{n\times n}$, $\bM^\Gamma_{i,j} = \int_{\Gamma_h} \varphi_j \varphi_i \, ds$, $i,j \in \{1,..,n\}$ we can write $\bc = \bM^{\Gamma} \be$ for every coefficient vector $\be \in \rr^n$ which represents a discrete function which is (constant) one on $\Gamma_h$. Note that $\bM^{\Gamma}$ has a large kernel, hence $\be \in \rr^n$ is not unique. However, a particular $\be$ can easily be constructed by considering the discrete function which is (constant) one on $\Omega_h^\Gamma$. 

In order to obtain a solvable linear system the compatibility condition must hold on the discrete level: $\int_{\Gamma_h} f_h \d{s_h} = \langle \bff , \be \rangle = 0$. Due to geometrical discretization errors it is not inherited from the corresponding property of the continuous problem \eqref{eq:weak-LB}. We proceed as suggested in Remark \ref{daterror}. Given an initial approximation $\tilde{f}$ on $\Gamma_h$ ($\tilde{\bff}$ with $\tilde{\bff}_i = \int_{\Gamma_h} \tilde{f} \varphi_i \d{s_h}$, $i\in\{1,..,n\}$), we define $f_h$ as in \eqref{extens} and let $ \bff = \tilde{\bff} - \frac{ \langle \tilde{\bff}, \be \rangle }{ \langle \bc , \be \rangle } \bc$.
Note that we have $f_h(v) = 0$ for every function $v$ which is constant on $\Gamma_h$, i.e. % $\tilde{\bff} \perp \mathrm{ker}(\bA)$.
$\bff \in \mathrm{range}(\bS)$. 
% We further note that every discrete function is constant on $\Gamma_h$ iff it is constant on $\Omega_\Theta^\Gamma$.\\
% {\bf AR: Q : is this always true? Related remark: the mass matrix $\bM_\Gamma$ can be singular (in special cases)?}\\
% $ = \{ \bx \in \rr^n, \bM^\Gamma \bx = \alpha \bc, \alpha \in \rr \}$.

To solve the constrained linear system we consider the uniquely solvable problem
\[
  \tilde \bS  \bu = \bff, \quad \text{with} ~~ \tilde \bS := \bS + \gamma \bc  \bc^T.
% = \bff + \frac{ \langle \bff, \be \rangle }{ \langle \bc , \be \rangle } \bc
\]
Here, we choose $\gamma = ( \sum_{i=1}^n \mathrm{diag}(\bS)_i ) / ( \sum_{i=1}^n \bc_i^2 )$ to approximately match the scaling of both terms. 
$\tilde{\bS}$ is symmetric positive definite and the solution of this system is unique and fulfils the equation $\bS  \bu = \bff$ and the constraint $ \langle \bc , \bu \rangle = 0$. To solve the system we apply a standard conjugate gradient method with diagonal preconditioning.
\end{remark}
% We note that we could also solve the system $\tilde{\bS} \cdot \bu = \bff$ which would result in solutions which fulfil $\bS \cdot \bu = \tilde{\bff}$ and $\bc^T \cdot \bu = \frac{\langle \bff, \be\rangle}{\gamma \langle \bc, \be \rangle}$. 

\subsection{Laplace--Beltrami equation on a toroidal surface} \label{sec:torus}

\begin{figure}[h!]
  \vspace*{-0.2cm}
  \includegraphics[trim=5cm 0.0cm 5cm 0.0cm, clip=true, 
                   height=0.25\textwidth]{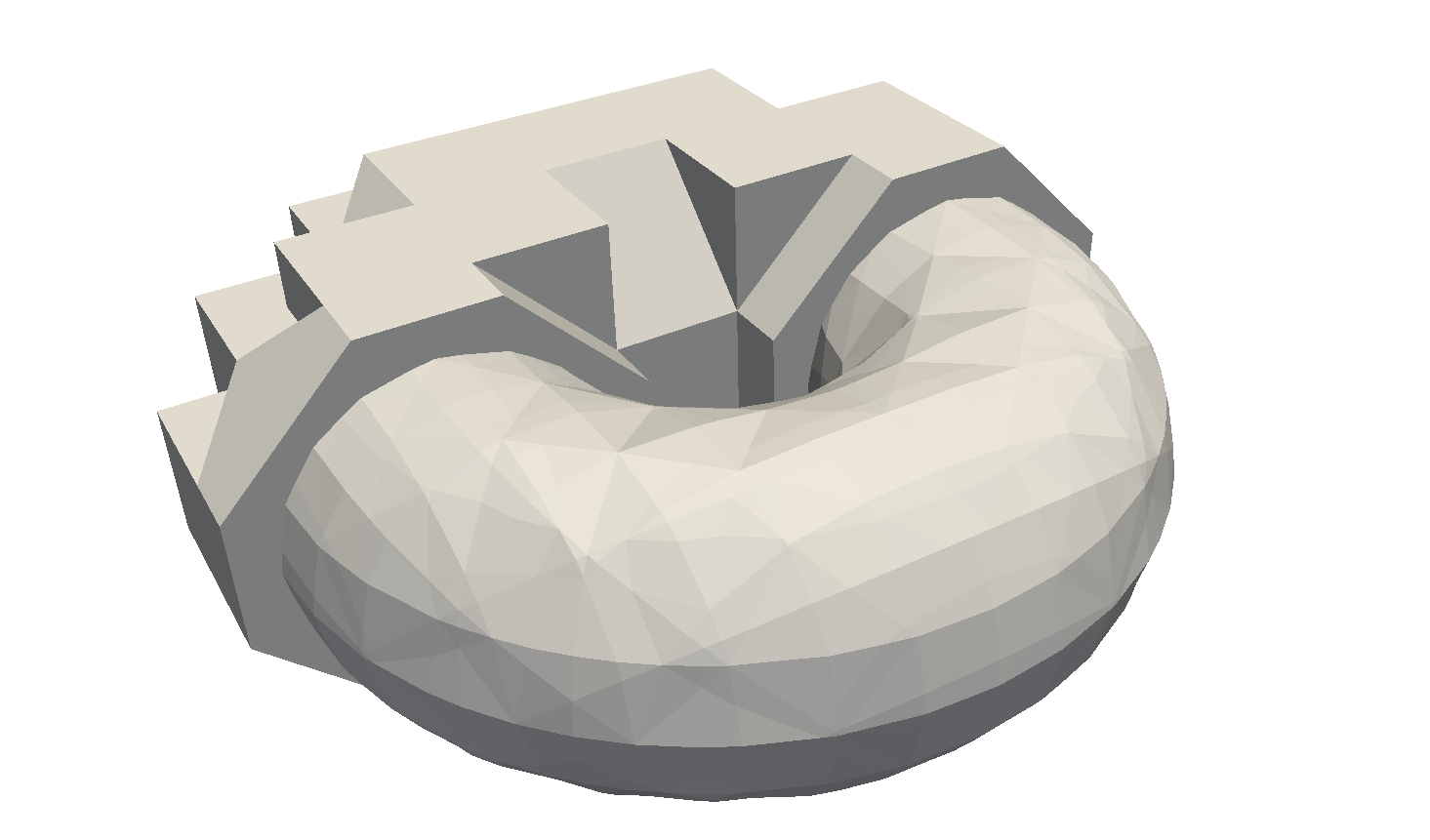} 
                   \hspace*{-0.7cm}
  \includegraphics[trim=5cm 0.0cm 5cm 0.0cm, clip=true, 
                   height=0.25\textwidth]{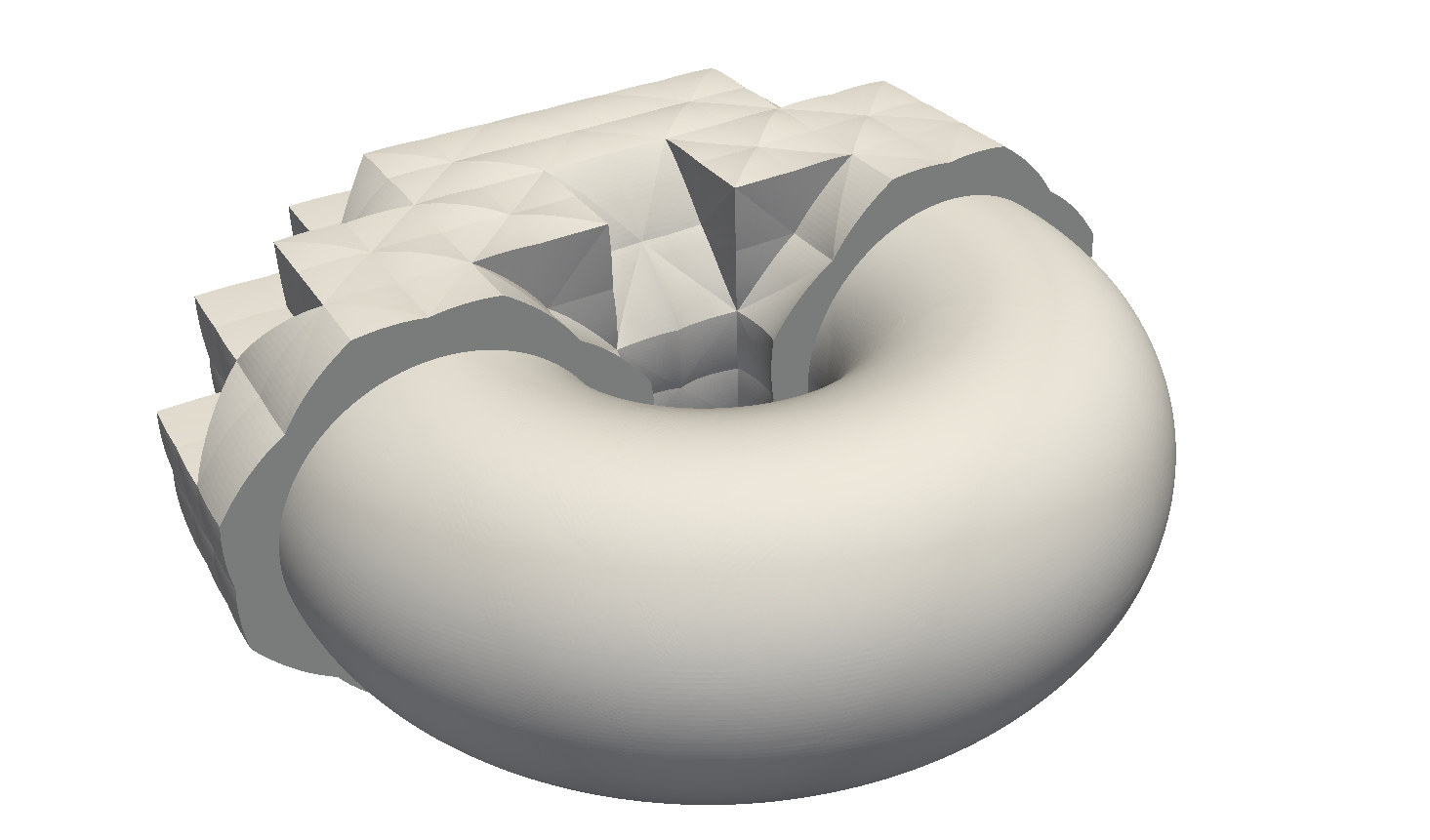} 
                   \hspace*{-0.7cm}
  \includegraphics[trim=5cm 0.0cm 5cm 0.0cm, clip=true, 
                   height=0.25\textwidth]{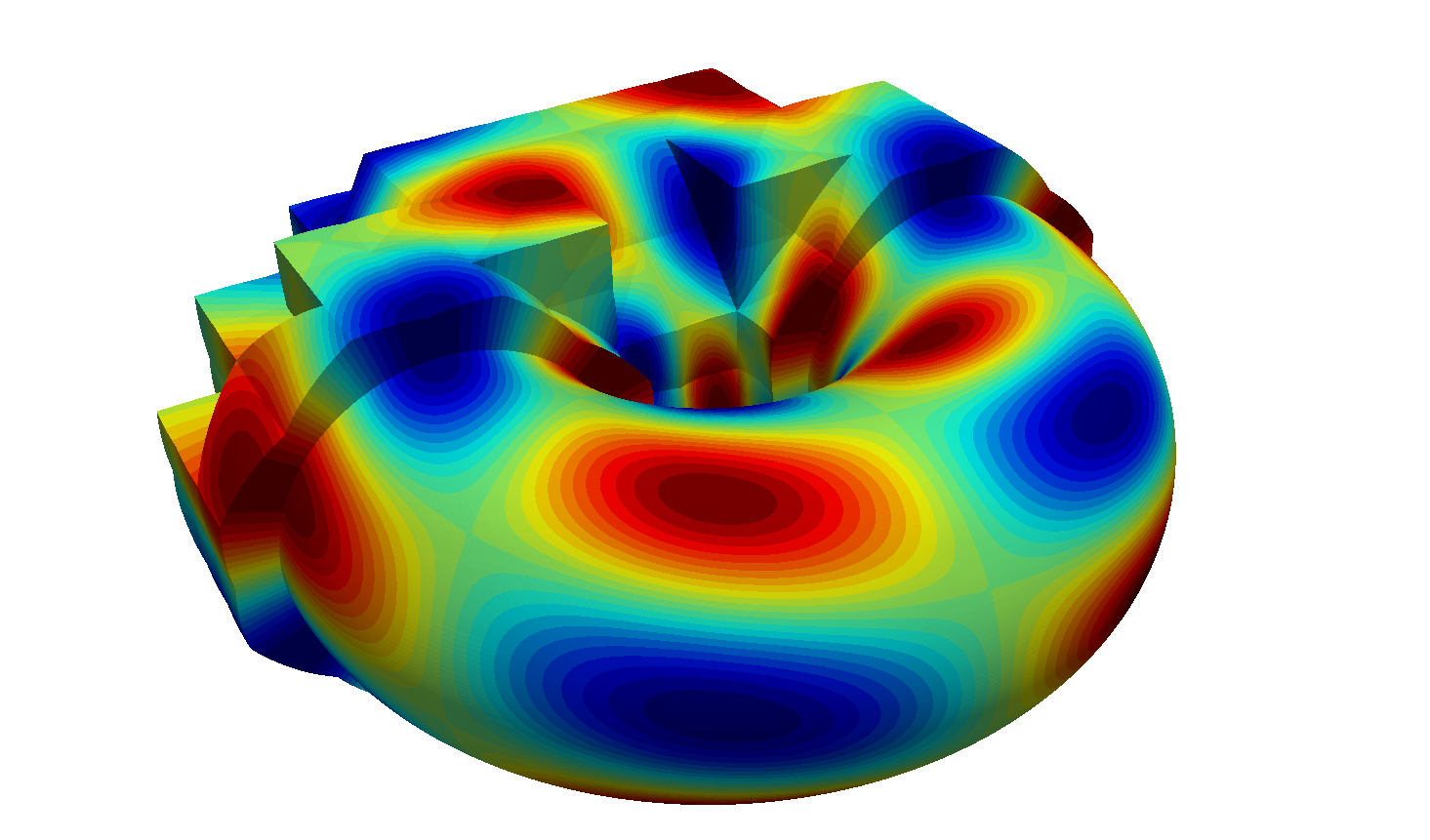} 
                   \hspace*{-0.7cm}
  \includegraphics[trim=0cm 0.0cm 0cm 0.0cm, clip=true, 
                   height=0.25\textwidth]{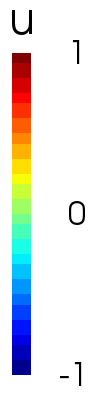} 
  \vspace*{-0.1cm}
  % \vspace*{-0.4cm}
  \caption{Numerical solution on the initial mesh for $k=3$ for the example in section \ref{sec:torus}: $\mT^\Gamma$ and $\Gammalin$ (left), $\Theta_h(\mT^\Gamma)$ and $\Gamma_h$ (center), $u_h \in \Vkz$ (right). }
  \label{fig:torus} 
\end{figure}

We consider an example from \cite{grande2014highordersjna} and apply the discretization described above with the normal derivative stabilization. The surface is a torus prescribed by the level set function $\phi$, $\Gamma = \{x \in \Omega~|~\phi(x) = 0\}$ with
\[
 \phi(x) = \left(x_3^2 + \left( \left(x_1^2+x_2^2\right)^\frac12 - R \right)^2\right)^\frac12 - r, \quad R=1, r=0.6.
\]
The surface is embedded in the domain $\Omega=[-2,2]^3$, and the solution is given as $u(x) = \sin(3 \varphi) \cos(3 \theta + \varphi)$ where $(\varphi,\theta)$ are the angles describing a surface parametrization, cf. \cite{grande2014highordersjna} for details. 
The right-hand side function $f$ is taken consistent with the solution $u$ and $\tilde{f}_h$ is chosen as the natural extension of this $f$. Then $f_h$ is constructed as described above, see remark \ref{sec:solvelinsys}.
Note that $u$ and $f$ have mean value zero on $\Gamma$ while $u_h$ and $f_h$ have mean value zero on $\Gamma_h$.
We start from a structured $16 \times 16 \times 16$ mesh ($h \approx \frac18$) and repeatedly apply uniform refinements (at the interface). 
In Figure \ref{fig:torus} the initial mesh is shown along with the surface approximations $\Gammalin$, $\Gamma_h$ and the discrete solution for $k=3$.
We investigate the behavior of the following quantities under mesh refinement. As a measure of the geometrical approximation quality we take $e_{dist} := \mathrm{dist}(\Gamma_h,\Gamma)$.
%$= \Vert \phi \Vert_{\infty,\Gamma_h}$ where we exploited the fact that $\phi$ is a signed distance function.
We further investigate the convergence of the errors 
$e_{L^2} := \Vert u^e - u_h \Vert_{L^2(\Gamma_h)}$, 
$e_{H^1}^t := \Vert \nabla_{\Gamma_h} (u^e - u_h) \Vert_{L^2(\Gamma_h)}$ and 
$e_{H^1}^n := \Vert \nabla u_h \cdot n \Vert_{L^2(\Gamma_h)}$. 
Here $u^e$ is the constant extension of $u$ along the normals of $\Gamma$.
In contrast to the stabilization term $s_h(\cdot,\cdot)$ 
the error measure $e_{H^1}^n$ is evaluated on the (discrete) surface $\Gamma_h$.
Finally, we also collect the number of CG iterations $N_{its}$ necessary to reduce the initial residual by a factor of $\num{1e-9}$.

We carry out the numerical experiment for the cases $\rho_s \sim h^{-1}$ and $\rho_s \sim h$, $k \in \{1,..,5\}$ and apply mesh refinements up to meshes with around a million unknowns.
%There is a constant in $\rho_s$ which can be chosen freely (independent of $h$).
In the numerical experiments we find that $\rho_s = h^{-1}$ gives much better results than $\rho_s = h$ in the sense that in the latter case we observe a strong dependence of the iteration number ($N_{its} \gtrsim k^2$) on the polynomial degree $k$. As a remedy we introduce a factor \emph{independent of $h$} into $\rho_s\sim h$. From a small test series we find that $\rho_s =   k^4 h$ gives results which are more robust with respect to variations in $k$.
At this point, we have no mathematical justification for the choice of the factor $k^4$. Note  that in our analysis  the dependence of the constants in the estimates on $k$  has not been considered. The results of the numerical experiments are displayed in Table \ref{rho_s1}.

\begin{table}[h!]
  \vspace*{-0.1cm}
  \footnotesize
  \hfill
  \begin{tabular}{l@{}r@{~~}l@{}r@{~~}l@{}r@{~~}l@{}r@{~}r}
    \toprule
    \multicolumn{9}{c}{$\rho_s \sim h^{-1}$} \vspace*{-0.02cm} \\
\midrule
$e_{dist}$&(eoc)&$e_{L^2}$&(eoc)&$e_{H^1}^t$&(eoc)&$e_{H^1}^n$&(eoc)&$N_{its}$ \\ %&$N_{its}^\ast$&$N_{d.o.f.}$ \\ %\multicolumn{2}{c}{$N_{d.o.f.}$}\\
\midrule
    \multicolumn{9}{l}{$k=1$ \quad ($912 - 900K$ unknowns)} \vspace*{-0.02cm} \\
    \midrule
\num{2.90286e-02} &              & \num{7.19179e-01} &              & \num{7.59399e+00} &              & \num{2.97248e+00} &              &       71  \\
\num{7.98578e-03} & ({\bf  1.9}) & \num{2.22706e-01} & ({\bf  1.7}) & \num{3.76102e+00} & ({\bf  1.0}) & \num{1.98229e+00} & ({\bf  0.6}) &      118  \\
\num{1.83402e-03} & ({\bf  2.1}) & \num{5.05149e-02} & ({\bf  2.1}) & \num{1.88546e+00} & ({\bf  1.0}) & \num{9.86263e-01} & ({\bf  1.0}) &      229  \\
\num{4.87394e-04} & ({\bf  1.9}) & \num{1.28923e-02} & ({\bf  2.0}) & \num{9.53279e-01} & ({\bf  1.0}) & \num{4.95408e-01} & ({\bf  1.0}) &      442  \\
\num{1.24949e-04} & ({\bf  2.0}) & \num{3.32666e-03} & ({\bf  2.0}) & \num{4.77926e-01} & ({\bf  1.0}) & \num{2.48855e-01} & ({\bf  1.0}) &      849  \\
\num{3.10248e-05} & ({\bf  2.0}) & \num{9.46600e-04} & ({\bf  1.8}) & \num{2.38365e-01} & ({\bf  1.0}) & \num{1.24845e-01} & ({\bf  1.0}) &     1652  \\
\midrule                                                                                                                                   
    \multicolumn{9}{l}{$k=2$ \quad ($5.3K - 5.3M$ unknowns)} \vspace*{-0.02cm} \\                                                                   
\midrule
\num{3.85000e-03} &              & \num{5.36388e-02} &              & \num{1.36773e+00} &              & \num{1.22060e+00} &              &      130  \\
\num{4.46928e-04} & ({\bf  3.1}) & \num{8.35991e-03} & ({\bf  2.7}) & \num{4.28053e-01} & ({\bf  1.7}) & \num{3.05451e-01} & ({\bf  2.0}) &      181  \\
\num{5.92903e-05} & ({\bf  2.9}) & \num{1.11528e-03} & ({\bf  2.9}) & \num{1.10183e-01} & ({\bf  2.0}) & \num{7.65697e-02} & ({\bf  2.0}) &      326  \\
\num{7.99525e-06} & ({\bf  2.9}) & \num{1.55728e-04} & ({\bf  2.8}) & \num{2.96749e-02} & ({\bf  1.9}) & \num{2.08708e-02} & ({\bf  1.9}) &      623  \\
\num{1.02274e-06} & ({\bf  3.0}) & \num{1.90417e-05} & ({\bf  3.0}) & \num{7.31939e-03} & ({\bf  2.0}) & \num{5.17451e-03} & ({\bf  2.0}) &     1178  \\
\num{1.26318e-07} & ({\bf  3.0}) & \num{2.36268e-06} & ({\bf  3.0}) & \num{1.81892e-03} & ({\bf  2.0}) & \num{1.28007e-03} & ({\bf  2.0}) &     2275  \\
\midrule                                                                                                                              
    \multicolumn{9}{l}{$k=3$ \quad ($16K - 4M$ unknowns)} \vspace*{-0.02cm} \\
    \midrule
\num{1.24052e-03} &              & \num{1.30372e-02} &              & \num{4.30460e-01} &              & \num{4.30498e-01} &              &      263  \\
\num{8.16528e-05} & ({\bf  3.9}) & \num{7.29942e-04} & ({\bf  4.2}) & \num{4.98292e-02} & ({\bf  3.1}) & \num{5.11916e-02} & ({\bf  3.1}) &      344  \\
\num{5.07904e-06} & ({\bf  4.0}) & \num{4.62651e-05} & ({\bf  4.0}) & \num{6.45915e-03} & ({\bf  2.9}) & \num{5.55089e-03} & ({\bf  3.2}) &      429  \\
\num{6.45444e-07} & ({\bf  3.0}) & \num{3.14062e-06} & ({\bf  3.9}) & \num{8.66504e-04} & ({\bf  2.9}) & \num{7.60004e-04} & ({\bf  2.9}) &      768  \\
\num{4.35992e-08} & ({\bf  3.9}) & \num{1.85714e-07} & ({\bf  4.1}) & \num{1.04391e-04} & ({\bf  3.1}) & \num{9.13136e-05} & ({\bf  3.1}) &     1420  \\
\midrule
    \multicolumn{9}{l}{$k=4$ \quad ($35K - 8.9M$ unknowns)} \vspace*{-0.02cm} \\ 
    \midrule
\num{6.51429e-04} &              & \num{1.44593e-03} &              & \num{6.53102e-02} &              & \num{1.20937e-01} &              &      528  \\
\num{1.11642e-05} & ({\bf  5.9}) & \num{5.04099e-05} & ({\bf  4.8}) & \num{4.85311e-03} & ({\bf  3.8}) & \num{5.99704e-03} & ({\bf  4.3}) &      600  \\
\num{4.34349e-07} & ({\bf  4.7}) & \num{1.79374e-06} & ({\bf  4.8}) & \num{3.22690e-04} & ({\bf  3.9}) & \num{3.49953e-04} & ({\bf  4.1}) &      681  \\
\num{1.59668e-08} & ({\bf  4.8}) & \num{8.19386e-08} & ({\bf  4.5}) & \num{2.70109e-05} & ({\bf  3.6}) & \num{2.50651e-05} & ({\bf  3.8}) &      945  \\
\num{5.45729e-10} & ({\bf  4.9}) & \num{2.58168e-09} & ({\bf  5.0}) & \num{1.49586e-06} & ({\bf  4.2}) & \num{1.53323e-06} & ({\bf  4.0}) &     1613  \\
\midrule                                                                                                                              
    \multicolumn{9}{l}{$k=5$ \quad ($66K - 1M$ unknowns)} \vspace*{-0.02cm} \\                                                     
    \midrule
\num{9.68761e-04} &              & \num{4.40866e-04} &              & \num{3.32631e-02} &              & \num{3.60224e-02} &              &     1071  \\
\num{1.18222e-06} & ({\bf  9.7}) & \num{6.04831e-06} & ({\bf  6.2}) & \num{6.06451e-04} & ({\bf  5.8}) & \num{8.39301e-04} & ({\bf  5.4}) &     1236  \\
\num{2.52701e-08} & ({\bf  5.5}) & \num{9.07531e-08} & ({\bf  6.1}) & \num{1.91119e-05} & ({\bf  5.0}) & \num{2.47323e-05} & ({\bf  5.1}) &     1312  \\
\num{7.26693e-10} & ({\bf  5.1}) & \num{2.40375e-09} & ({\bf  5.2}) & \num{7.59469e-07} & ({\bf  4.7}) & \num{9.86219e-07} & ({\bf  4.6}) &     1676  \\
    \bottomrule
  \end{tabular}
  \hfill
  \hfill
  \begin{tabular}{l@{}r@{~~}r}
    \toprule
    \multicolumn{3}{c}{$\rho_s \sim h$} \vspace*{-0.02cm} \\
    \midrule
    $e_{H^1}^n$&(eoc)&$N_{its}$ \\
    \midrule
    \vspace{-0.02cm}\\
    \midrule
    \num{5.45763e+00} &              &       69  \\
    \num{3.98106e+00} & ({\bf  0.5}) &      121  \\
    \num{3.10661e+00} & ({\bf  0.4}) &      248  \\
    \num{2.59565e+00} & ({\bf  0.3}) &      473  \\
    \num{2.39875e+00} & ({\bf  0.1}) &      937  \\
    \num{2.32129e+00} & ({\bf  0.0}) &     1872  \\
    \midrule
    \vspace{-0.02cm}\\
    \midrule
    \num{1.21806e+00} &              &      130  \\
    \num{4.01971e-01} & ({\bf  1.6}) &      192  \\
    \num{1.29864e-01} & ({\bf  1.6}) &      378  \\
    \num{4.53554e-02} & ({\bf  1.5}) &      730  \\
    \num{1.37695e-02} & ({\bf  1.7}) &     1543  \\
    \num{3.37812e-03} & ({\bf  2.0}) &     3118  \\
    \midrule
    \vspace{-0.02cm}\\
    \midrule
    \num{2.48577e-01} &              &      273  \\
    \num{4.93405e-02} & ({\bf  2.3}) &      335  \\
    \num{6.68864e-03} & ({\bf  2.9}) &      530  \\
    \num{1.17189e-03} & ({\bf  2.5}) &     1011  \\
    \num{1.71228e-04} & ({\bf  2.8}) &     2073  \\
    \midrule
    \vspace{-0.02cm}\\
    \midrule
    \num{8.44460e-02} &              &      482  \\
    \num{4.97260e-03} & ({\bf  4.1}) &      464  \\
    \num{3.49648e-04} & ({\bf  3.8}) &      680  \\
    \num{3.21935e-05} & ({\bf  3.4}) &     1261  \\
    \num{2.55372e-06} & ({\bf  3.7}) &     2582  \\
    \midrule
    \vspace{-0.02cm}\\
    \midrule
    \num{2.07080e-02} &              &      935  \\
    \num{6.89020e-04} & ({\bf  4.9}) &     1016  \\
    \num{2.28371e-05} & ({\bf  4.9}) &     1098  \\
    \num{1.05918e-06} & ({\bf  4.4}) &     1836  \\
    \bottomrule
  \end{tabular}
  \hfill
  \caption{Results for the example in section \ref{sec:torus} with $\rho_s = h^{-1}$ (left) and $\rho_s = hk^4$ (right).} 
  \label{rho_s1}
  \vspace*{-0.45cm}
\end{table}

As predicted in \eqref{resdist} we observe $\mathcal{O}(h^{k+1})$-convergence for the geometrical error measure $e_{dist}$. We note that the initial triangulation is sufficiently fine to guarantee the mesh regularity of the deformed meshes at all refinement levels.

With respect to the error measures $e_{H^1}^t$ and $e_{L^2}$ we only display the results for $\rho_s = h^{-1}$ in Table \ref{rho_s1} because the differences between the different stabilization scalings in those error measures are only marginal. For $e_{H^1}^t$
we observe $\mathcal{O}(h^k)$-convergence which is in agreement with the prediction of Theorem \ref{mainthm}. For $e_{L^2}$, we observe the optimal rate $\mathcal{O}(h^{k+1})$, but have no a priori analysis for this, yet.

The previous error measures are essentially unaffected by the choice of the stabilization scaling. This is different for the number of iterations $N_{its}$ and the error measure $e_{H^1}^n$. 
For $k=1$ we observe that $e_{H^1}^n$ does not convergence for $\rho_s\sim h$ while it converges with order one for $\rho_s \sim h^{-1}$. 
In the higher order case, $k\geq 1$, the difference in the results is much smaller. For both scalings we observe at least $e_{H^1}^n \lesssim h^{k-1/2} \rho_s^{-1/2}$. The results even indicate a convergence order $k$ in both cases, although this is more pronounced for $\rho_s \sim h^{-1}$ than for $\rho_s \sim h$.

The iteration counts for both scalings increase linearly with the mesh size for sufficiently fine meshes which is in agreement with the condition number bound in Theorem \ref{coro:condition-number}. On coarser grids and for increasing order $k$ the numbers of iterations stagnate before the asymptotic regime starts and the iteration counts grow linearly. 

\begin{remark}[No stabilization, $s_h(\cdot,\cdot) \equiv 0$] \rm
It is known that for $k=1$  stabilization is in general not necessary for satisfactory iteration numbers in the CG method, provided diagonal preconditioning is applied, cf. \cite{reusken10matrixprop}. Accordingly, we repeated the previous numerical experiment with $k=1$ and $\rho_s=0$. We obtain similar results for $e_{H^1}^t$ and $e_{L^2}$, whereas $e_{H^1}^n$ does not converge (with similar errors as for $k=1$ and $\rho_s \sim h$). The iteration counts are larger (95, 175, 360, 793, 1470, 2890), but also increase linearly with $h$.
In our experience, for moderate orders, $k=2,3$, a discretization with $\rho_s=0$  \emph{often}  yields results for $e_{H^1}^t$, $e_{L^2}$ and $N_{its}$ which are similar to those obtained with stabilization. However, there is no control on $e_{H^1}^n$ and, more importantly, \emph{sometimes} the linear solver fails to converge or the iteration numbers are very high (even with diagonal preconditioning). For even higher order, $k \geq 4$, in general the (diagonally preconditioned) CG solver does not converge for $\rho_s=0$.
\end{remark}

\section{Conclusion and outlook} \label{sec:conclusion}
We introduced and analyzed a higher order iso\-para\-metric trace FEM. The higher discretization accuracy is obtained by using an isoparametric mapping of the volume mesh, based on a high order approximation of the level set function. The resulting trace finite element method is easy to implement. We presented an error analysis of this method and derived optimal order $H^1(\Gamma)$-norm error bounds. A second main topic of this paper is a unified analysis of several stabilization methods for this class of surface finite element methods. The recently developed normal derivative volume stabilization method is analyzed.  This method is able to control the condition number of the stiffness matrix also for the case of higher order discretizations.
 
We mention a few topics which we consider to be of interest for future research. Firstly, the derivation of an optimal order $L^2$-error bound has not been investigated, yet. We think that most ingredients needed for such an anlysis are available from this paper, e.g. the $\mathcal{O}(h^{k+1})$-consistency-bound in Lemma~\ref{lem:conserr}. A second, much more challenging, topic is the extension of the higher trace finite element technique presented in this paper to the class of PDEs on \emph{evolving} surfaces. It may be possible to extend the isoparametric mapping technique to a space-time setting and then combine it with the trace space-time method for discretization of PDEs on evolving surfaces given in \cite{olshanskii14spacetime,olshanskii14spacetimeanalysis}. As a final topic we mention the extension of the higher order discretization technique presented in this paper to coupled bulk-surface problems. 
\appendix

\section{Proof of Lemma~\ref{lem:dh}} \label{sec:proof:lem:dh}
\begin{proof}
First we prove the bound in \eqref{eq:psihbound}.
For $T \in \TGamma$ we consider the function $F(x,y) = \mathcal{E}_T \phi_h (x + y G_h(x))-\hphi(x)$ for $(x,y) \in T \times (-\alpha_0 h, \alpha_0 h)$, with $G_h:= \nabla \phi_h$. 
From \cite[Lemma 3.2]{CLARH1} we know that there exists a $h_0 > 0$ so that for all $0< h < h_0$ the function $d_h(x) = y(x)$ solves $F(x,y(x)) = 0$ on $T$. 
Since $\hphi, \mathcal{E}_T \phi_h$ and $G_h$ are polynomials and hence 
$F \in C^{\infty}(T \times (-\alpha_0 h, \alpha_0 h))$ it follows from the implicit function theorem that $y \in C^{\infty}(T)$.
Due to $D^{\alpha} \hphi = 0$ for $|\alpha|>1$ we have $\Vert D^{\alpha} \hphi \Vert_{\infty,T} \lesssim \Vert \phi \Vert_{H^{2,\infty}(T)}$ independent of $\alpha$. Using the extended element $U(T) := \{x+w\mid x \in T, |w| \leq 2 \alpha_0 h \}$ and the continuity of the polynomial extension operator $\mathcal{E}_T$ we have with $l = |\alpha|\leq k+1$:
\begin{equation} \label{implbound} \begin{split}
  \Vert D^{\alpha}_{(x,y)} F \Vert_{\infty,T\times(-\alpha_0 h, \alpha_0 h)} &
  \lesssim \Vert \mathcal{E}_T \phi_h \Vert_{H^{l,\infty}(U(T))} \Vert G_h\|_{H^{l,\infty}(T)} \\
  & \lesssim \Vert \phi_h \Vert_{H^{l,\infty}(T)}\Vert\phi_h\Vert_{H^{l+1,\infty}(T)} \lesssim\Vert\phi\Vert_{H^{l+1,\infty}(T)}^2 \lesssim 1. 
\end{split} \end{equation}
Differentiating $F(x,y(x))=0$ yields, for $|\alpha|=1$:
\begin{equation} \label{impl1}
D^{\alpha} y(x) = - D_y F(x,y(x))^{-1} D_x^{\alpha} F(x,y(x)) = -A(x) D_x^{\alpha} F(x,y(x)).
\end{equation}
with $A(x) = S(x)^{-1}$, $S(x) = D_y F(x,y(x)) = \nabla \mathcal{E}_T \phi_h (x + y G_h(x))^T \nabla \phi_h \in [c_0,c_1]$ with $c_0,c_1 > 0$ independent of $h, x, T$. Differentiating $S(x) A(x) = 1$ yields 
\begin{equation} \label{impl2}
D^{\alpha} A(x) = - A(x)^2 D^{\alpha} S(x), \quad |\alpha|=1.
\end{equation}
From \eqref{impl1} and \eqref{impl2} we deduce that $| D^{\alpha} y(x)|,~|\alpha|=l$, can be bounded in terms of $\vert A(x) \vert$ and $| D^{\alpha}_{(x,y)} F(x,y(x))|,~|\alpha|\leq l$. Combining this with \eqref{implbound} gives the first bound in \eqref{eq:psihbound}.
From $\Vert G_h \Vert_{H^{l,\infty}(T)} \lesssim \Vert \phi_h \Vert_{H^{l+1,\infty}(T)} \lesssim \Vert \phi \Vert_{H^{l+1,\infty}(T)}\lesssim 1$
and the first bound we obtain the second bound in \eqref{eq:psihbound}.

For \eqref{eq:jumpbound} we consider an interior facet $F \in \mFGamma$ with neighboring tetrahedra denoted by $T_1,T_2\in \mTGamma$. 
We set $d_h^i = d_h|_{T_i}$ and $G_h^i = G_h|_{T_i}$ for $i=1,2$. As $\hphi$ is continuous we have
\begin{equation}\label{jumpe}
\mathcal{E}_{T_1} \phi_h(x+d_h^1(x) G_h^1(x)) - \mathcal{E}_{T_2} \phi_h(x+d_h^2(x) G_h^2(x)) = 0 \quad \text{for all}~ x \in F.
\end{equation}
Using \eqref{eq:lsetdist} we obtain for $x \in F$ and with $G:=\nabla \phi$,
\begin{align*}
  |d_h^1(x) - d_h^2(x)| & \sim | \phi(x+d_h^1(x)G(x)) - \phi(x+d_h^2(x)G(x)) | \\
   & \lesssim | \phi(x+d_h^1(x)G_h^1(x)) - \phi(x+d_h^2(x)G_h^2(x)) | \\
 & + \sum_{i=1}^2| \phi(x+d_h^i(x)G(x)) - \phi(x+d_h^i(x)G_h^i(x)) |. %\stackrel{\eqref{Taylor}\&\eqref{jumpe}}{\lesssim} h^{k+1} \\ 
\end{align*}
For the sum we use the regularity of $\phi$ in $U$, \eqref{eq:lsetdist} and the estimates for $G_h^i - G$ (cf. \cite[Lemma 3.1]{CLARH1}): 
\[
| \phi(x+d_h^i(x)G(x)) - \phi(x+d_h^i(x)G_h^i(x)) | \lesssim |d_h^i(x)| \Vert G_h^i(x) - G(x) \Vert_2 \lesssim h^{k+2}.
\]
For the other term we use $y_i:=x+d_h^i(x)G_h^i(x)$ and \eqref{jumpe}:
\[
 |\phi(y_1)-\phi(y_2)| \leq |\phi(y_1)-\mathcal{E}_{T_1} \phi_h(y_1)| + |\phi(y_2)-\mathcal{E}_{T_2} \phi_h(y_2)|. \vspace*{-0.1cm}
\]
The two terms on the right-hand side can be bounded by $\mathcal{O}(h^{k+1})$ using Taylor expansion arguments, cf. \cite[Proof of Lemma 3.2]{CLARH1}, which concludes the proof of the first bound in \eqref{eq:jumpbound}. 

Finally, we consider $\jump{\Psi_h} = \jump{d_h G_h} = \jump{d_h} \average{G_h} + \average{d_h} \jump{G_h}$ (with $\average{a} := \frac{a_1+a_2}{2}$).
From \cite[Lemma 3.1]{CLARH1} and the assumed regularity of $\phi$ we have (uniform in $x$) $|\average{G_h}| \lesssim 1$, $|\jump{G_h}| \lesssim h^k$ and with the first bound in \eqref{eq:jumpbound} and \cite[Lemma 3.2]{CLARH1} we have $|\jump{d_h}| \lesssim h^{k+1}$ and $|\average{d_h}| \lesssim h^2$. Together this proves \eqref{eq:jumpbound}.
\end{proof}

\bibliographystyle{siam}
\bibliography{literature}

\end{document}